\documentclass[a4paper,10pt]{amsart}

\usepackage{hyperref,amssymb,enumerate,url,tikz,multirow}

\title{Projection Operators in the Weihrauch lattice}
\author{Guido Gherardi}
\address{Dipartimento di Filosofia,
    Universit\`{a} di Bologna}
\email{guido.gherardi@unibo.it}

\author{Alberto Marcone}
\address{Dipartimento di Scienze Matematiche, Informatiche e Fisiche,
    Universit\`{a} di Udine}
\email{alberto.marcone@uniud.it}
\urladdr{http://users.dimi.uniud.it/~alberto.marcone/}
\thanks{Marcone's research partially supported by PRIN 2012 Grant \lq\lq Logica, Modelli
e Insiemi\rq\rq\ and by the departmental PRID funding \lq\lq HiWei --- The
higher levels of the Weihrauch hierarchy\rq\rq.}

\author{Arno Pauly}
\address{Department of Computer Science, Swansea University}
\email{a.m.pauly@swansea.ac.uk}

\date{\today}

\newcommand{\set}[2]{\left\{\,{#1} \,:\, {#2}\,\right\}}

\newcommand{\Can}{\ensuremath{{2^\IN}}}
\newcommand{\Bai}{\ensuremath{{\IN^\IN}}}
\newcommand{\Sie}{\mathbb S}

\newcommand{\mto}{\rightrightarrows}

\newcommand{\var}{\varepsilon}
\def\IN{\mathbb{N}}

\def\IQ{\mathbb{Q}}
\def\IR{\mathbb{R}}

\def\AA{\mathcal{A}}
\def\KK{\mathcal{K}}
\def\FF{\mathcal{F}}

\def\toto{\rightrightarrows}

\def\sbsq{\subseteq}
\def\eps{\emptyset}

\def\otto{\leftrightarrow}
\def\To{\Rightarrow}

\def\Otto{\Longleftrightarrow}

\def\fa{\forall}
\def\ex{\exists}

\def\dom{\mathrm{dom}}
\def\range{\mathrm{range}}

\def\id{\mathrm{id}}
\def\lim{\mathrm{lim}}

\newcommand{\Proj}{\operatorname{Proj}}
\newcommand{\ProjC}{\Proj}
\newcommand{\ProjK}{\Proj\K}
\newcommand{\ProjvarC}{\var\text{-}\Proj}
\newcommand{\ProjvarK}{\var\text{-}\Proj\K}

\newcommand{\C}{\operatorname{C}}
\newcommand{\K}{\operatorname{K}}

\newcommand{\LPO}{\operatorname{LPO}}
\newcommand{\LLPO}{\operatorname{LLPO}}
\newcommand{\WKL}{\operatorname{WKL}}
\newcommand{\BWT}{\operatorname{BWT}}
\newcommand{\UBWT}{\operatorname{UBWT}}

\newcommand{\Sort}{\operatorname{Sort}}

\def\leqW{\mathop{\leq_{\mathrm{W}}}}

\def\equivW{\mathop{\equiv_{\mathrm{W}}}}

\def\leqsW{\mathop{\leq_{\mathrm{sW}}}}

\def\equivsW{\mathop{\equiv_{\mathrm{sW}}}}

\def\nleqW{\mathop{\not\leq_{\mathrm{W}}}}

\def\lW{\mathop{<_{\mathrm{W}}}}

\def\nW{\mathop{|_{\mathrm{W}}}}

\newtheorem{lemma}{Lemma}[section]
\newtheorem{theorem}[lemma]{Theorem}
\newtheorem{corollary}[lemma]{Corollary}
\newtheorem{fact}[lemma]{Fact}
\newtheorem{proposition}[lemma]{Proposition}

\theoremstyle{definition}
\newtheorem{definition}[lemma]{Definition}

\theoremstyle{remark}
\newtheorem{remark}[lemma]{Remark}

\begin{document}

\begin{abstract}
In this paper we study, for $n \geq 1$, the projection operators over
$\IR^n$, that is the multi-valued functions that associate to $x \in \IR^n$
and $A \subseteq \IR^n$ closed, the points of $A$ which are closest to $x$.
We also deal with approximate projections, where we content ourselves with
points of $A$ which are almost the closest to $x$. We use the tools of
Weihrauch reducibility to classify these operators depending on the
representation of $A$ and the dimension $n$. It turns out that, depending
on the representation of the closed sets and the dimension of the space,
the projection and approximate projection operators characterize some of
the most fundamental computational classes in the Weihrauch lattice.
\end{abstract}

\maketitle

\tableofcontents


\section{Introduction}\label{sec:introduction}

Projecting a point over a non-empty subset of a Euclidean space is an
operation deeply grounded in our geometrical intuition of the spatial
continuum and has many important applications in higher mathematics. More
precisely, given $x \in \IR^n$ and $A \subseteq \IR^n$ we seek $y \in A$ such
that $d(x,y) = d(x,A)$ (when $A$ is closed, such a $y$ does exist, although
it might not be unique). In this paper we show that the intuitive, even
empirical, naturalness of this problem leads to multi-valued function
realizing some well-known levels of incomputability.

We work in the Weihrauch lattice, which has become a widespread tool to
classify the level of incomputability of mathematical problems from several
branches of classical mathematics since \cite{HB}. Intuitively, given two
(multi-valued) functions $f$ and $g$ on represented spaces, $f$ is
\emph{Weihrauch reducible} to $g$ if $f$ can be computed by $g$, with
computable translations from $\dom(f)$ to $\dom(g)$, and, viceversa, from
$\range(g)$ to $\range(f)$, allowed. (More details are in \S\ref{ssec:Wred}
below.)

Recall that in this approach mathematical objects are encoded by sequences of
infinite length whose information is based on the topological properties of
the underlying spaces. For example, $x \in \IR^n$ (for a fixed $n\geq 1$) is
naturally represented by an effective Cauchy sequence in $\IQ^n$ converging
to $x$. But if we want to project $x$ onto a subset of $\IR^n$ we of course
also need a suitable encoding for the set. We focus our investigation on
closed sets (although we will also mention the special case of compact sets)
and their standard topologies. The so-called \emph{negative} representation
for closed sets is based on the lower Fell topology $\AA_-(\IR^n)$ and
consists in enumerating an open cover of the complement. In contrast, their
\emph{positive} representation is based on the upper Fell topology
$\AA_+(\IR^n)$ and consists, for nonempty closed sets, in enumerating dense
sequences of points in them. Finally, the \emph{total} representation,
corresponding to the Fell topology, is obtained by combining both kinds of
information (\cite{Sch}). More details about these representations will be
given in \S\ref{ssec:repr} below. We thus obtain different projection
operators, depending on the representation chosen for the closed
set.\smallskip

In the literature (\cite{BD10, Neumann}) it has been proved that the
projection operators, for some metric spaces and closed sets
with optimal conditions (such as convexity, boundedness of $A$, uniqueness of
the solution) are computable. But what happens when such optimal
conditions fail? It is not surprising that the problem is then no longer
computably solvable for any of the above mentioned representations on closed
sets, so the goal becomes the classification, depending on the type of
information involved, of the corresponding degrees of incomputability in the
Weihrauch lattice.\smallskip

In many concrete applications one may be content already with approximations
of arbitrarily accurate precision. In other words, we investigate the
computational complexity of operators selecting points $y\in A$ such that
$d(x,y) \leq (1+\varepsilon) d(x,A)$ for some fixed $\varepsilon>0$.
Intuitively, we expect that the loss of accuracy results in a simpler
computational complexity. We indeed prove that these operators are simpler
than their exact counterparts, but still not computable for negative and
positive representations of closed sets. In contrast, the approximate
projection operators with total information are computable.\smallskip

\begin{table}
\begin{tabular}{|c|c|c|l|c|}
\hline
\textbf{Proj.} & \textbf{Repr.} & \textbf{Dim.} & \textbf{Weihrauch degree} & \textbf{Reference}\\
\hline
\hline
\multirow{6}*{Exact}
& \multirow{2}*{negative}
& $n=1$ & $\equivW \BWT_2\times\lim$ & \ref{BWT2Xlimneg}\\
\cline{3-5}
& & $n \geq 2$ & $\equivW \BWT_\IR$ & \ref{equiv-BWT}\\
\cline{2-5}
& \multirow{2}*{positive}
& $n=1$ & $\equivW \BWT_2\times\lim$ & \ref{BWT2Xlimpos}\\
\cline{3-5}
& & $n\geq 2$ & $\equivW \BWT_\IR$ & \ref{equiv+BWT}\\
\cline{2-5}
& \multirow{2}*{total} & $n=1$ & $\equivW \LLPO$ & \ref{ProjCequivLLPO}\\
\cline{3-5}
& & $n\geq 2$ & $\equivW \WKL$ & \ref{ProjCequivWKL}\\
\hline
\multirow{3}*{Approx.} & negative & $n\geq 1$ & $\equivW \C_\IR$ & \ref{approxequivCR}\\
\cline{2-5}
& positive & $n\geq 1$ & $\equivW \Sort$ & \ref{approxequivsort} \\
\cline{2-5}
& total & $n\geq 1$ & computable & \ref{totalapprox} \\
\hline
\end{tabular}\medskip

\caption{Summary of main results: the first column indicates the kind of
projection, the second column the representation of the closed sets, the
third one the dimension of the Euclidean space $\IR^n$, the fourth one our
results, and the last one the references to the results in this paper where
the results are proved.\label{table}}
\end{table}
Table \ref{table} summarizes our main results. Quite surprisingly, it turns
out that in most cases the (approximate) projection operators are Weihrauch
complete with respect to some fundamental computational class which is
represented in the last column by its emblematic representative, already
studied in the literature, and defined in \S\ref{ssec:milestones} below. In
other words, the notion of projection allows us, by varying the kind of
projection, the representation of the closed sets, and the dimension of the
space, to characterize some of the most fundamental degrees in the Weihrauch
lattice. In all cases, we obtain a characterization in terms of previously
studied Weihrauch degrees.

It is also remarkable that, as far as exact projections are concerned,
negative or positive information for closed sets can be used interchangeably,
as this has no effect in the classification obtained with respect to any
given dimension $n\geq 1$. The difference between negative and positive
information only arises when approximations are allowed. Here, the relevant
degrees are even incomparable (Corollary \ref{corr:incomparable}).
\smallskip

To see that the approximate projections are of practical importance we
suggest a concrete application, which is actually the original motivation for
our research. The Whitney Extension Theorem was originally proved in
\cite{Whitney} and, roughly speaking, generalizes the well-known
Urysohn-Tietze Extension Lemma to the case of differentiable functions. An
expository paper on modern developments concerning the Whitney Extension Theorem
and its generalizations is \cite{Feff}. By using approximate projections in
place of the exact ones originally used in the proof of the Whitney Extension
Theorem as exposed in the classical textbook \cite{Stein}\footnote{Stein
himself suggests on p.172 the possibility of using approximate projections,
although our use probably is not the one he was foreseeing.}, we sketch the
computability of this theorem. Full details of this result are postponed to a
forthcoming paper (\cite{compWhitney}).\medskip

We now explain the organization of the paper. In Section \ref{sec:notation}
we give a brief introduction to computable analysis, introduce the
representations we will be using throughout the paper, and recall Weihrauch
reducibility and some milestones in the Weihrauch lattice. Section
\ref{sec:Sort} provides a new characterization of the Weihrauch degree of the
function $\Sort$, introduced by Neumann and Pauly (\cite{NP}). Sections
\ref{sec:exact} and \ref{sec:approx} are the core of the paper and are
devoted respectively to the exact and approximated projection operators:
after defining them we prove the results summarized in Table \ref{table}. In
Section \ref{sec:Whitney} we briefly sketch the application of approximate
projections to the Whitney Extension Theorem.

\section{Computable analysis: notation and terminology}\label{sec:notation}

This Section recalls basic definitions and terminology of computable analysis
and of Weihrauch reducibilities (see \cite{Weihrauch} for a self-contained
introduction to the subject). The reader familiar with the topics can safely
skip it and refer back to this section as needed.\smallskip

We work in the framework is the so called Type-2 Theory of Effectivity (TTE),
which finds a systematic foundation in \cite{Wei00} and provides a realistic
and flexible model of computation. The salient features of TTE Turing
machines are that they work on infinite sequences of bits and that no
correction is allowed on the output. A partial function $F: \sbsq \Bai \to
\Bai$ is computable if it is in computed by some TTE Turing machine. An
immediate consequence of the restraint concerning the output is that all
computable functions are continuous.

\subsection{Representations}\label{ssec:repr}
To extend the notion of computability to functions between spaces different
from \Bai\ we need the notion of representation. Recall that a
\emph{representation} $\sigma_X$ of a set $X$ is a surjective function
$\sigma_X: \sbsq \Bai \to X$, and in this case we say that the pair
$(X,\sigma_X)$ is a \emph{represented space}. If $x \in X$ a
\emph{$\sigma_X$-name} for $x$ is any $p \in \Bai$ such that $\sigma_X(p) =
x$. By routine syntactic pairing techniques it is straightforward to obtain
representations for finite and countably infinite product of represented
spaces.

Given represented spaces $(X, \sigma_X)$ and $(Y, \sigma_Y)$ and a partial
multi-valued function $f: \sbsq X \mto Y$, we say that $F:\sbsq \Bai \to
\Bai$ is a \emph{$(\sigma_X, \sigma_Y)$-realizer} of $f$ (and write $F \vdash
f$) if $\sigma_Y(F(p)) \in f(\sigma_X(p))$, for all $p \in \dom(f \circ
\sigma_X)$. We can now say that a function between represented spaces is
\emph{computable} if it has a computable realizer.

For representations $\sigma_X$ and $\sigma'_X$ of the same set $X$, we say
that $\sigma_X$ is \emph{computably reducible} to $\sigma'_X$ (we write
$\sigma_X \leq_c \sigma_X'$) if there is a computable $F:\sbsq \Bai \to \Bai$
such that for every $p \in \dom(\sigma_X)$ we have $\sigma_X(p)=
\sigma'_X(F(p))$. If $\sigma_X \leq_c \sigma'_X$ and $\sigma'_X \leq_c
\sigma_X$, the two representations are \emph{computably equivalent}
($\sigma_X \equiv_c \sigma_X$).

The general notion of representation is too broad for practical purposes.
Concretely, representations are associated to the final topologies they
induce on the represented space, and usually \emph{admissible
representations} for $T_0$-spaces are considered. Such representations are
those that make the use of realizers meaningful: if $X$ and $Y$ are
admissibly represented, a single valued $f:\sbsq X\to Y$ is continuous if and
only if it admits a continuous realizer $F:\sbsq\Bai\to\Bai$ with respect to
the Baire topology (see \cite{Wei00} and \cite{Paadm} for introductions to
the theory of admissible representations).

An important example is the Cauchy representation which is admissible with
respect to the topology of a separable (computable) metric space.

\begin{definition}[Computable metric spaces]\label{metric}
A \emph{computable metric space} is a triple $(X,d,\alpha)$, where $d$ is a
metric on $X$, $\alpha: \IN \to X$ is a dense sequence in $X$, and $d \circ
(\alpha \times \alpha)$ is a computable double sequence in $\IR$. We then
represent $X$ by the \emph{Cauchy representation} $\delta_X: \sbsq \Bai \to
X$, defined by
\begin{align*}
p\in\dom(\delta_X) & \Otto (\fa i)(\fa j\geq i) \: d(\alpha(p(i)),\alpha(p(j)))\leq 2^{-i}; \\
\delta_X(p)=x & \Otto \lim_{n\to\infty}\alpha(p(n))=x.
\end{align*}
When $\delta_X(p)=x$ we say that $(\alpha(p(i)))_i$ is an \emph{effective
Cauchy sequence}, and that it \emph{converges effectively} to $x$.
\end{definition}

Notice that with this representation, the metric $d: X \times X \to \IR$ is
computable.

A particularly important example is provided by the Euclidean spaces $\IR^n$,
which are computable metric spaces when we fix a function
$\alpha:\IN\to\IQ^n$ enumerating in an effective way $\IQ^n$. Here
$d:\IR^n\times\IR^n\to\IR$ is the usual Euclidean metric.

By using the same effective numbering $\alpha:\IN\to\IQ$, there are other
ways to represent real numbers, by changing the underlying topology over
$\IR$. The representation $\rho_>$ is given by $\rho_>(p):=x$ iff
$n\in\range(p)\Otto \alpha(p(n))<x$, and analogously $\rho_>$ is given by
$\rho_>(p):=x$ iff $n\in\range(p)\Otto \alpha(p(n))>x$. These two
representations are admissible with respect to the topologies $\IR_<$ and
$\IR_>$ whose open sets are of the form $]x,\infty[$ and $]-\infty,x[$
respectively (see \cite{Wei00} for more details).

Given a computable metric space $(X,\delta_X)$ we can effectively enumerate
the open balls with center in $\range(\alpha)$ and rational radius in an
obvious way using a computable pairing function: to $k= \langle n,m\rangle
\in \IN$ we associate the open ball $B_k := B(\alpha(n),q_m)$, where
$(q_m)_m$ is a standard enumeration of the nonnegative rational numbers
(notice that $B(\alpha(n),q_m)=\eps$ when $q_m=0$). We call these sets
\emph{open basic balls}. We denote the closed ball $\set{x\in
X}{d(\alpha(n),x)\leq q_m}$ by $\overline B(\alpha(n),q_m)$ or $\overline
B_k$ (notice that in general this is not the same as the closure
$\overline{B(\alpha(n),q_m)}=\overline{B_k}$ of $B_k$, although in $\IR^n$
they coincide).

\begin{definition}[Closed set representations]\label{closed}
Let $(X,\delta_X)$ be a computable metric space.

By $\AA_-(X)$ we denote the hyperspace of closed subsets of $X$ equipped with
the \emph{negative information representation} $\psi_X^-: \Bai\to\set{A\sbsq
X}{A \text{ is closed in } X}$ such that
$$\psi_X^-(p)=A \Otto A=X\setminus\bigcup_{i \in \range(p)} B_{k_i}.$$\smallskip

By $\AA_+(X)$ we denote the hyperspace of closed subsets of $X$ equipped with
the \emph{positive information representation} $\psi_X^+: \sbsq\Bai \to
\set{A\sbsq X}{A \text{ is closed in } X}$ such that
$$\psi^+(p)=A\Otto (\fa i) \Big(i\in\range(p)\otto B_{k_i} \cap A \neq \eps
\Big).$$\smallskip

Finally, by $\AA(X)$ we denote the hyperspace of closed subsets of $X$
equipped with the \emph{total information representation} $\psi_X = \psi_X^-
\wedge \psi_X^+$, that is
$$\psi_X(\langle p_0,p_1\rangle) = A \Otto \psi_X^-(p_0) = \psi_X^+(p_1) = A.$$
\end{definition}

It is clear from Definitions \ref{metric} and \ref{closed} that we can view
$A$ as an element of $\AA_-(X)$ if and only if we can semi-decide whether
$x\notin A$ for every $x \in X$. This means that to show that (a name for)
some $A \in \AA_-(X)$ can be computed from some input $z$ it suffices to give
a definition of $A$ by a $\Pi^0_1$ formula with parameter $z$.

It is well known that the operations $\cap, \cup: \AA_-(X) \times \AA_-(X)
\to \AA_-(X)$ are computable, as well as $\cup: \AA(X)_+ \times \AA(X)_+ \to
\AA_+(X)$.

Closed sets with positive information are also known in the literature as
\emph{overt sets} (see \cite{Paadm} for a discussion of nomenclature).

\begin{remark}\label{rem:psi+}
By \cite[Theorems 3.7 and 3.8]{BP03}, in every complete computable metric
space $(X,\delta_X)$, the positive information representation for
\emph{nonempty} closed sets is equivalent to the representation which assigns
to a name $p:=\langle p_0,p_1,p_2,\dots\rangle \in \dom(\delta_X)^\IN$ the
set $\overline{\set{\delta_X(p_n)}{n \in \IN}}$. See also \cite[Lemma
5.1.10]{Wei00} for the case $X=\IR^n$.

As for the negative information, in the Euclidean space this is equivalent to
the representation encoding a closed $A\sbsq\IR^n$ by enumerating all $k$
such that $A \cap \overline B_k=\eps$ (\cite[Lemma 5.1.10]{Wei00}).
\end{remark}

We are also interested in representing the space of the compact subsets of a
fixed computable metric space.

\begin{definition}[Compact set representations]
Let $(X,\delta_X)$ be a computable metric space. By $\KK_-(X)$ we denote the
hyperspace of compact subsets of $X$ equipped with the \emph{negative
information representation} $\kappa_X^-:\sbsq\Bai\to\set{K\sbsq X}{K \text{
is compact in } X}$ such that:
$$
\kappa_X^-(\langle k_0,\dots, k_{j-1} \rangle p)=K \Otto
\psi_X^-(p)=K \land K \sbsq\bigcup_{i<j}\overline B_{k_i} \land (\fa i<j)\:B_{k_i}\neq X.
$$
Analogously, one defines the hyperspace $\KK_+(X)$ of compact subsets of $X$
equipped with the \emph{positive information representation} $\kappa_X^+$,
and the hyperspace $\KK(X)$ of compact subsets of $X$ equipped with the
\emph{total information representation} $\kappa_X$, by replacing $\psi_X^-$
with $\psi_X^+$ and $\psi_X$, respectively.
\end{definition}

\begin{remark}\label{compeucl}
In the case of the Euclidean space $\IR^n$, the balls
$B_{k_0},\dots,B_{k_{j-1}}$ can be more simply replaced by a single ball
$B(0,N)$, for $N\in\IN$, satisfying $K\sbsq \overline B(0,N)$ (in agreement
with \cite[Definition 5.2.1]{Wei00}).
\end{remark}

Abstracting from the purely syntactic elements, representations often denote
objects by enumerating sequences of objects. Therefore one is often allowed
to skip the annoying linguistic aspects by describing the represented element
directly through the corresponding sequence of objects. For instance, we see
a point $x$ in a metric space directly as $x = \lim_{i\to\infty} x[i]$,
where, for all $i$, $x[i]=\alpha(p(i))\in X$ with respect to some given
Cauchy-name $p$ of $x$. Analogously, we can describe a closed set
$A\in\AA_-(X)$ directly as $A:=X\setminus\bigcup_{i\in\IN} B_i$, by meaning
that $B_i$ (which really should be $B_{k_i}$) is the $i$-th rational open
ball enumerated by some $\psi^-_X$-name of $A$.

\subsection{Weihrauch reducibility}\label{ssec:Wred}
The original definition of Weihrauch reducibility between functions over
represented spaces is due to Weihrauch in an unpublished report from 1992,
and in the next decade the notion was explored in several thesis by some of
Weihrauch's students. The authors \cite{HB} extended Weihrauch reducibility
to multi-valued functions. Let $f:\sbsq X \mto Y$ and $g:\sbsq Z \mto W$ be
partial multi-valued functions between represented spaces. We say that
\emph{$f$ is Weihrauch reducible to $g$}, and write $f \leqW g$, if there are
computable $H:\sbsq\Bai\times\Bai\to\IN$ and $K:\sbsq\Bai\to\Bai$ such that
$H(\id,GK)\vdash f$ whenever $G\vdash g$ (here $\id:\Bai\to\Bai$ is the
identity function on Baire space).

The intuition behind the definition is that $f \leqW g$ means that the
problem of computing $f$ can be computably and uniformly solved by using in
each instance a single computation of $g$: $K$ modifies (each name for) the
input of $f$ to feed it to $g$, while $H$, using also the original input,
transforms (any name for) the output of $g$ into (a name for) the correct
output of $f$. Another characterization of Weihrauch reducibility is provided
by the fact that $f \leqW g$ if and only if there is a Turing machine that
computes $f$ using $g$ as an oracle exactly once during its infinite
computation \cite{TW11}.

A direct consequence of the definition of Weihrauch reducibility is the
following \emph{Invariance Principle}: $f\leqW g$ implies that for any given
$\sigma_X$-name $p$ of some $x\in\dom(f)$ there is some $y\in f(x)$ with a
$\sigma_Y$-name $q$ such that $q\leq_T p\oplus GK(p)$ (here $\leq_T$ denotes
the usual Turing reducibility). In other words, $GK(p)$ provides an upper
bound for the computational complexity of (some element in) $f(x)$.

The relation $\leqW$ is reflexive and transitive and induces an equivalence
relation denoted by $\equivW$. The partial order on the sets of
$\equivW$-equivalence classes (called \emph{Weihrauch degrees}) is a
distributive bounded lattice \cite{BG11a,Pa10} with several natural and
useful algebraic operations \cite{BW}. As usual, we use $f \lW g$ to denote
$f \leqW g$ and $g \nleqW f$, and $f \nW g$ to denote $f \nleqW g$ and $g
\nleqW f$.

The Weihrauch lattice can be used as a tool for comparing multi-valued
functions arising from theorems from different areas of mathematics, once the
theorems are translated into mathematical problems on represented spaces.
This line of research has blossomed in the last few years \cite{HB, BG11,
BG11a, BdBP, Pa10, Pa10a, BW, BGH15, BGH15a, DDHMS, LasVegas, ADSS} and this
paper contributes to it by classifying the projection operators.

In some cases, one can prove the reducibility of $f$ to $g$ by using a
computable $K$ that does not access to the original input, that is, we have
$KGH \vdash f$ whenever $G \vdash g$. In this case we say that \emph{$f$ is
strongly Weihrauch reducible to $g$} and write $f \leqsW g$. We then use
$\equivsW$ for the induced equivalence relation.

We notice that $f\leqsW g\Longrightarrow f\leqW g$ always hold, whereas
$f\leqW g\Longrightarrow f\leqsW g$ holds when $g$ is a \emph{cylinder}, that
is $g\equivsW g\times\id$, where id is the identity function on the Baire
space.

\subsection{Some milestones in the Weihrauch lattice}\label{ssec:milestones}
Multi-valued functions $f:\sbsq X\toto Y$ can be seen as problems: given
$x\in\dom(f)$, find a $y\in f(x)$. The algebraic structure of the Weihrauch
lattice can provide a useful tool to determine the computational complexity
of fundamental mathematical problems which are not computable, at least in
the standard TTE-model. A typical example is to find the derivative $f'$ of a
differentiable real function. Other paradigmatic problems are those provided
by fundamental theorems of classical mathematics. The idea here is to see a
statement of the form $(\fa x\in X)(\varphi(x)\to(\ex y\in Y)\:\psi(x,y))$ as
defining the multi-valued function $f:\sbsq X\toto Y$ with domain $\set{x \in
X}{\varphi(x)}$ and $f(x) = \set{y\in Y}{\psi(x,y)}$. It turns out that the
Weihrauch degree of many mathematical problems can be evaluated with the help
of few operators. Relevant for our work are the following:
\begin{itemize}

\item $\LPO:\Bai\to\{0,1\}$, the computational version of the \emph{limited
    principle of omniscience} of constructive mathematics defined by
    $\LPO(p)=0$ if $p=0^\IN$ and $\LPO(p)=1$ if  $p\neq0^\IN$;
\item $\LLPO:\sbsq\Bai\times\Bai\toto\{0,1\}$, the computational version of
    the \emph{lesser limited principle of omniscience} of constructive
    mathematics defined by $i \in \LLPO(p_0,p_1)$ iff $p_i=0^\IN$, where
    for at most one $j \in \{0,1\}$ and one $n\in\IN$ it holds $p_j(n) \neq
    0$;
\item $\WKL:\sbsq \mathsf{Tr} \toto \Can, T\mapsto[T]$, the \emph{Weak
    K\"onig's Lemma operator}, mapping each infinite binary tree to its
    infinite paths; here a tree $T \sbsq 2^{<\IN}$ is represented by its
    characteristic function $t \in \Can$, that is, $t(n)=1$ iff $w_n \in T$
    for a recursive enumeration $w_0, w_1, w_2, \dots$ of all finite binary
    words;
\item $\C_X:\sbsq\AA_-(X)\toto X,A\mapsto A$, the \emph{closed choice
    operators}, selecting members from any given non-empty closed set
    (encoded by negative information) in a computable metric space $X$;
\item $\K_X:\sbsq\KK_-(X)\toto X,K\mapsto K$, the \emph{compact choice
    operators}, selecting members from any given non-empty compact set
    (encoded by negative information) in a computable metric space $X$;
\item $\lim:\sbsq(\Bai)^\IN\to\Bai,(p_n)_n\to\lim_{n\to\infty}p_n$, for
    every convergent sequence $(p_n)_n$ in the Baire space;
\item $\BWT_X:\sbsq X^\IN\toto X$, the \emph{Bolzano-Weierstra\ss \:
    Theorem} operators, that maps every sequence with compact range in a
    computable metric space $X$ to its accumulation points.
\end{itemize}

For instance, the problem of finding the derivative $f'$ of $f$ is Weihrauch
equivalent to $\lim$. As for $\C_X$, $\K_X$, and $\BWT_X$, we obtain very
important cases when we set $X=\IN$ or $X=\IR$. For example, the
(contrapositive of) the Baire Category Theorem is Weihrauch equivalent to
$\C_\IN$. See \cite{Weihrauch} for a general overview of this program of
classification of mathematical problems and for further references.

It is well known that $\LLPO\equivsW\C_2$ (where $2=\{0,1\}$) and
$\WKL\equivsW\K_\IR$.

A multi-valued function $f$ is called \emph{non deterministically computable}
if $f\leqW\WKL$, \emph{computable with finitely many mind changes} if
$f\leqW\C_\IN$, and \emph{limit computable} if $f\leq\lim$. This terminology
arises from the non standard models of computation that make such $f$
computable. For instance, $f$ is computable with finitely many mind changes
if it can be computed by a non standard TTE-machine that is allowed to revise
the output with the restraint that only finitely many corrections can occur.
See \cite{Weihrauch} for more details and further references.

Some degrees can be seen as the \emph{parallelization} or \emph{composition}
of other degrees. The parallelization of $\widehat f:\sbsq X^\IN\toto Y^\IN$
of $f:\sbsq X\to Y$ is defined as $\widehat f((x)_n):=(f(x)_n)$. We have for
example $\WKL\equivsW\widehat{\LLPO}$ and $\lim\equivsW\widehat\LPO$. It is
known that parallelization is a closure operator, that is $f\leqW\widehat f$,
$f\leqW g\To\widehat f\leqW\widehat g$, and $\widehat{\widehat
f}\equivW\widehat f$.

The composition of multi-valued functions is defined so that the range of the
first function not necessarily has to be included in the domain of the second
function. Intuitively, some computational transformation is allowed so that
the two spaces can match. It is easier to define such \emph{compositional
product} as an operation on degrees\footnote{The introduction of the
compositional product as a specific function is more involved (see
\cite{algebraicWeihrauch}, \cite[Definition 5.3]{Weihrauch}) and not relevant
for this paper.} by
$$
g*f:=\max\set{g_0\circ f_0}{g_0\leqW g, f_0\leqW f}.
$$
Here the leftmost occurrences of $f$ and $g$ must be understood as denoting
the corresponding degrees, and the maximum as a degree defined by the partial
order induced on the Weihrauch degrees by $\leqW$. (Notice that the Weihrauch
lattice is not complete, but the $\max$ above always exists by
\cite[Corollary 18]{algebraicWeihrauch}, \cite[Theorem 5.2]{Weihrauch}.)  It
holds then
\begin{center}
$\BWT_\IR \equivW \BWT_{\Can} \equivW \WKL*\lim$ and $\C_\IR \equivW
\WKL*\C_\IN \equivW \C_\IN*\WKL$.
\end{center}
These equivalences justify the following
terminology: $f \leqW \BWT_\IR$ is said to be \emph{non deterministically
limit computable} and $f \leqW \C_\IR$ is said to be \emph{non
deterministically computable with finitely many mind changes}.

Finally, some degrees can be seen as \emph{jumps} of others. Given a
multi-valued function $f:\sbsq X\toto Y$ on represented spaces
$(X,\delta_X),(Y,\delta_Y)$, the jump $f':\sbsq X\toto Y$ of $f$ coincides
with $f$ but the representation of $X$ is weakened into the representation
$\delta'_X$, where $\dom(\delta'_X):=\set{\langle
p_0,p_1,p_2,...\rangle\in\IN^\IN}{\lim_{i\to\infty}
p_i\downarrow\in\dom(\delta_X)}$ and $\delta'_X(\langle p_0,p_1,p_2, \dots
\rangle)=x\in X$ iff $\delta_X(\lim_{i\to\infty}p_i)=x$. It holds then
$\BWT_2\equivsW\LLPO'$ and $\BWT_\IR\equivsW\WKL'$.

Notice that $\lim$ is a cylinder, hence $f\leqsW\lim \Otto f\leqW\lim$, and
the same holds for $\WKL$ and $\BWT_\IR$.

\section{The functions $\Sort$ and $\min^-_{\omega+1}$}\label{sec:Sort}

In \cite{NP} Neumann and Pauly introduced the function $\Sort: \Can \to \Can$
defined as
$$
\Sort(p):=
\begin{cases}
0^n1^\IN & \textrm{if $p$ contains exactly $n$ occurrences of 0}\\
0^\IN & \textrm{if $p$ contains infinitely many occurrences of 0}
\end{cases}
$$
Our results support the importance of this function, so that one might see it
as a candidate for a new possible milestone in the Weihrauch lattice. To this
end we first show that $\Sort$ is strongly Weihrauch equivalent to another
natural function.

Consider the space
$$
\omega+1:=\bigcup_{n\in\IN}\{-2^{-n}\}\cup\{0\}.
$$
This space is seen as a subspace of the represented space $\IR$, hence its
members are represented as real numbers via Cauchy sequences of elements of
the dense set of rationals in $\omega+1$, i.e., $\omega+1$ itself.

It is easy to see that this Cauchy representation $\delta_{\omega+1}$ is
computably equivalent to the representation $\rho_{\omega+1}$ with
$\dom(\rho_{\omega+1})=\set{p\in\Bai}{(\exists^{\leq 1} i)p(i)\neq 0}$ and
$$
\rho_{\omega+1}(p)=
\begin{cases}
-2^{-n} & \textrm{if $p(n)\neq0$;}\\
0 & \textrm{if $(\fa i)p(i)=0$}
\end{cases}
.
$$
To see that $\rho_{\omega+1} \leq_c \delta_{\omega+1}$, take any
$\rho_{\omega+1}$-name $p$ of $x\in\omega+1$ and consider the Cauchy sequence
$(x_n)_n$ such that $x_i:=2^{-i}$ if $p(j)=0$ for all $j\leq i$, and
$x_i:=2^{-j}$ if $p(j)\neq 0$ for a (unique) $j\leq i$. For the opposite
reduction, let $(x[n])_n$ converge effectively to $x$. To obtain a
$\rho_{\omega+1}$-name $p$ of $x$ just put $p(i)=0$ if $x[i+2]\neq 2^{-i}$
and $p(i)=1$ otherwise. Intuitively, according to the representation
$\rho_{\omega+1}$ a name of $x \in \omega+1$ is a (computable!) oracle that
for every $i\in\IN$ replies ``yes'' or ``no'' to the question ``is
$x=-2^{-i}$?''; if the answer is always ``no'' then $x= 0$.

It is often more convenient to represent $\omega+1$ by $\rho_{\omega+1}$.
However, when representing the space of closed subsets of $\omega+1$ we will
view $\omega+1$ as a computable metric space and use the standard enumeration
of the basic open balls $B(a,q)$, for $a\in\omega+1$ and $q$ a nonnegative
rational number, to obtain the representation $\psi^-_{\omega+1}$ of
$\AA_-(\omega+1)$.

As a subset of $\IR$, $\omega+1$ is also well-ordered by the usual order $<$.
The single valued function $\min^-_{\omega+1}: \sbsq \AA_-(\omega+1) \to
\omega+1$ mapping $A $ to $\min(A)$ is then defined for every $A \neq \eps$.

\begin{proposition}\label{sortomega}
$\Sort \equivsW \min^-_{\omega+1}$.
\end{proposition}
\begin{proof}
We first show that $\Sort \leqsW \min^-_{\omega+1}$. Given $p \in \Can$ we
construct a set $A \in \AA_-(\omega+1)$ that will provide us with the
necessary information to compute $\Sort(p)$. More precisely, we define $A :=
(\omega+1) \setminus \bigcup_{s\in\IN} B_s$, where $B_s :=
\set{-2^{-m}}{m<n}$ if $p(0), \dots, p(s)$ contains exactly $n$ occurrences
of $0$. Let now $r$ be a $\rho_{\omega+1}$-name of $\min(A)$. By
construction, $r(n)\neq 0$ if $0$ occurs exactly $n$ times in $p$. We obtain
then $q:=\Sort(p)$ as follows. We inspect $r$ and as long as $r(n)=0$, we let
$q(n)=0$, so that $q=0^\IN$ if $r(n)=0$ for all $n\in\IN$. As soon as we find
an $n$ such that $r(n)\neq 0$, then we let $q(m)=1$ for every $m\geq n$, so
that in the end $q=0^n1^\IN$.

To prove $\min^-_{\omega+1} \leqsW \Sort$ argue as follows. Let $A :=
(\omega+1) \setminus \bigcup_{n\in\IN}B_n \in \AA_-(\omega+1)$ be given as
input to $\min^-_{\omega+1}$. Our strategy consists simply in choosing at any
stage $s$ the smallest element of $\omega+1$ not contained in $B_0,\dots,B_s$
and we want to write an input $r\in\Bai$ for Sort that reflects our choice.
At stage 0 let then $x_0$ be the least element of $\omega+1$ not contained in
$B_0$. If this is 0, then we write $r(0)=0$. Otherwise, let it be $-2^{-n_0}$
for some $n_0$. Then we put $0^{n_0}1$ as initial segment of the input $r$ of
Sort. At stage $k+1$ we consider the sets $B_0, \dots, B_{k+1}$. Let
$x_{k+1}$ be the least element not contained in $\bigcup_{i\leq k+1}B_i$, and
let $w$ be the initial segment of $r$ obtained at stage $k$. If $x_{k+1}=0$,
then we let $r(|w|):=0$. Otherwise, if $x_{k-1}=-2^{-n_{k+1}}$ for some
$n_{k+1}$ we extend $w$ so to obtain a finite prefix $ww'1$ with $ww'$
containing exactly $n_{k+1}$ occurrences of $0$ (possibly $|w'|=0$). In the
end, by construction, $r$ contains exactly $n$ occurrences of 0 if
$\min(A)=-2^{-n}$ for some $n\in\IN$, and $r$ contains infinitely many $0$ if
$\min(A)=0$. We now inspect $\Sort(r)$ to compute a $\rho_{\omega+1}$-name
$q$ of $\min(A)$. Recall that $\Sort(r)=0^\IN$ if $r$ contains infinitely
many occurrences of $0$, that is, $\min(A)=0$, otherwise
$\Sort(r)=0^{n}1^\IN$ if $r$ contains exactly $n$ occurrences of $0$, that
is, $\min(A)=-2^{-n}$. Therefore, to obtain a correct $\rho_{\omega+1}$-name
$q$, we let $q(n):=\Sort(r)(n)$ as long as $\Sort(r)(n)=0$. If suddenly
$\Sort(r)(n)=1$, then we let $q(n):=1$ and $q(m)=0$ for all $m>n$. In this
way we obtain exactly the $\rho_{\omega+1}$-name of $\min(A)$.
\end{proof}

Using Proposition \ref{sortomega}, we now study the degree of $\Sort$ in more
detail. The following result is given already in Proposition 24 of \cite{NP}
but we give here a more direct proof of the same result in terms of
computability with finitely many mind changes using $\min^-_{\omega+1}$.

\begin{proposition}\label{CINsottomega}
$\C_\IN \lW \Sort$.
\end{proposition}
\begin{proof}
By Proposition \ref{sortomega} we can substitute $\Sort$ with
$\min^-_{\omega+1}$.

To prove $\C_\IN \leqW \min^-_{\omega+1}$, consider the operator $\min^-_\IN:
\sbsq \AA_-(\IN) \to \IN, A \mapsto \min(A)$, for $A \neq \eps$, which is
known to be Weihrauch equivalent to $\C_\IN$ by \cite[Lemma 2.3]{PFD}. We
obtain $\C_\IN \equivW \min^-_\IN \leqsW \min^-_{\omega+1}$ (for the
rightmost reduction observe that the map $A \mapsto \set{-2^{-n}}{n \in A}
\cup \{0\}$ from $\AA_-(\IN)$ to $\AA_-(\omega+1)$ is clearly computable).

To prove that $\min^-_{\omega+1} \nleqW \C_\IN$, we will show that
$\min^-_{\omega+1}$ is not computable with finitely many mind changes.

Let $p$ indeed be a $\psi^-_{\omega+1}$-name of a nonempty closed $A \sbsq
\omega+1$. The task is to output the minimal element in $A$. Suppose that $p$
lists only open balls of the type $\{-2^{-i}\}$ for various $i\in\IN$. If the
sequence encoded by $p$ will in the end contain every open ball of the form
$\{-2^{-i}\}$, the temporary choice of any element of the form $-2^{-k}$ will
sooner or later force us to select a larger candidate. In this case we obtain
the name of the correct output $0$ only after infinitely many mind changes.

We should therefore choose $0$ as the eventual output at some stage $s$, when
only a finite initial segment of $p$ has been read. However, this output is
incorrect if the sequence encoded by $p$ never mentions a specific element
$\{-2^{-m}\}$, which is possible on the basis of the finite initial segment
of $p$ read by the computation at sage $s$.
\end{proof}

The next result is also given in Proposition 24 of \cite{NP}:

\begin{proposition}\label{omegalim}
$\Sort \lW \lim$.
\end{proposition}
\begin{proof}
It is easy to see that $\Sort \leqsW \lim$. To prove that the opposite
reduction does not hold, apply the Invariance Principle: $\Sort$ has only
computable outputs, whereas $\lim$ maps some computable input to an
incomputable output.
\end{proof}

For the following results, we need the Bolzano Weierstra\ss\ operators
$\BWT_n$, with $n:=\{0,\dots,n-1\}$ for $n\geq 1$, and the operators
$\UBWT_X:\sbsq X^\IN\to X$, which are the restrictions of the operators
$\BWT_X$ to the sequences with compact range for which the accumulation point
is unique.

\begin{proposition}\label{omegaBWTn}
For every $n \geq 1$, $\min^-_{\omega+1} \nleqW \BWT_n$.
\end{proposition}
\begin{proof}
$\min^-_{\omega+1} \leqW \BWT_n$ would imply $\C_\IN \leqW \BWT_n$ by
Proposition \ref{CINsottomega}. Since $\C_\IN \equivW \UBWT_\IN$ by
\cite[Corollary 11.24]{BW} and obviously $\UBWT_{n+1} \leqW \UBWT_\IN$, this
would in turn imply $\UBWT_{n+1} \leqW \BWT_n$, which is impossible by
\cite[Proposition 13.9]{BW}.
\end{proof}

We now show that $\Sort$ is not non deterministically computable with
finitely many mind changes:

\begin{proposition}
\label{prop:sortcr}
$\Sort \nW \C_\IR$
\end{proposition}
\begin{proof}
Recall that an operation $f$ is \emph{non-uniformly computable}, if $f(x)$
contains a computable solution for all computable $x$. A Weihrauch degree $f$
is called \emph{low}, if $\lim*f \leqW \lim$. Both properties are
preserved downwards under Weihrauch reduction.

Notice that $\Sort$ is non-uniformly computable as all solutions are
computable. Moreover $\LPO*\Sort$ computes the characteristic
function of the $\Sigma^0_2$-complete set
$$\set{p \in \Bai}{\text{$p$
contains finitely many occurrences of } 0}.$$
Therefore $\LPO*\Sort\nleqW\lim$. Since  $\LPO\leqW\lim$ and $*$ is monotone,  it follows that $\Sort$ is not low. On the other hand, $\C_\IR$ is
low by \cite[Theorem 8.7]{BdBP}, but not non-uniformly computable because
$\WKL \leqW \C_\IR$ and there exist computable infinite binary trees without
computable infinite branches. The incomparability of $\Sort$ and $\C_\IR$
then follows immediately.
\end{proof}

\section{Exact projections operators}\label{sec:exact}

We start with the formal definition of the exact projection operators.

\begin{definition}
Given a metric space $X$, a point $x \in X$ and a nonempty set $A \subseteq
X$ we say that $y \in A$ is a \emph{projection point of $x$ onto $A$} if
$d(x,y) =d(x,A)$ (where, as usual, $d(x,A) = \inf \set{d(x,y)}{y \in A}$). In
other words, the projection points of $x$ onto $A$ are the points of $A$ with
minimal distance from $x$.
\end{definition}

Notice that if $x \in A$ then $x$ itself is the unique projection point of
$x$ onto $A$. Obviously, projection points of $x$ onto $A$ exist if and only
if the infimum in the definition of $d(x,A)$ is actually a minimum. We will
be mostly interested in the case where $X = \IR^n$ is a Euclidean space and
$A$ is closed; in this situation projection points of any $x \in \IR^n$ onto
$A$ do exist.

If $X$ is a computable metric space, projections points give rise to several
multi-valued functions, depending on the representation of $A \subseteq X$,
which we will always assume to be at least closed.

\begin{definition}
Given a computable metric space $X$ the \emph{(exact) negative, positive and
total closed projection operators on $X$} are the partial multi-valued
functions $\ProjC^-_X$, $\ProjC^+_X$ and $\ProjC_X$ which associate to every
$x \in X$ (with Cauchy representation) and every closed $A \neq \emptyset$
(with negative, positive and total representation, respectively) the set of
the projection points of $x$ onto $A$.

Thus $\ProjC^-_X: \sbsq X \times \AA_-(X) \toto X$, $\ProjC^+_X: \sbsq X
\times \AA_+(X) \toto X$, and $\ProjC_X: \sbsq X \times \AA(X) \toto X$.

The \emph{(exact) negative, positive and total projections operators for
compact sets} are defined by replacing $\AA_-(X)$, $\AA_+(X)$, and $\AA(X)$
with $\KK_-(X)$, $\KK_+(X)$, and $\KK(X)$ respectively. These are denoted
$\ProjK^-_X$, $\ProjK^+_X$, and $\ProjK_X$.
\end{definition}

The first obvious observation is that the negative projection operators
compute the corresponding choice operators.

\begin{fact}\label{kcredexact}
$\K_X \leqsW \ProjK^-_X$ and $\C_X \leqsW \ProjC^-_X$ for all computable
metric spaces $X$.
\end{fact}

Projections on compact sets are special cases of projections on closed sets.

\begin{fact}\label{kridc}
For all computable metric spaces $X$:
\begin{enumerate}
\item $\ProjK^-_X \leqsW \ProjC^-_X$,
\item $\ProjK^+_X \leqsW \ProjC^+_X$,
\item $\ProjK_X
\leqsW \ProjC_X$.
\end{enumerate}
\end{fact}
\begin{proof}
The proof follows immediately by the definition
of the representations.
\end{proof}

In some important cases, the inverse reduction holds as well. In the next
result, a computable metric space $X$ is \emph{computably compact} when it is
computable as a member of $\KK_-(X)$, that is, it has some computable
$\kappa_X^-$-name (or, equivalently, of $\KK(X)$).

\begin{fact}\label{kccompact}
For all computably compact metric spaces $X$:
\begin{enumerate}
\item $\ProjK^-_X \equivsW \ProjC^-_X$,
\item $\ProjK^+_X \equivsW \ProjC^+_X$,
\item $\ProjK_X\equivsW \ProjC_X$.
\end{enumerate}
\end{fact}
\begin{proof}
The inverse reductions of Fact \ref{kridc} can be obtained by fixing a finite
cover of $X$ by basic balls and use it to show that $id:\AA_-(X)\to\KK_-(X)$,
$id:\sbsq\AA_+(X)\to\KK_+(X)$, and $id:\sbsq\AA(X)\to\KK(X)$ are computable.
\end{proof}

\begin{theorem}\label{CequivK}
For $n \geq 1$:
\begin{enumerate}
\item $\ProjK^-_{\IR^n} \equivW \ProjC^-_{\IR^n}$,
\item $\ProjK^+_{\IR^n} \equivsW\ProjC^+_{\IR^n}$,
\item $\ProjK_{\IR^n} \equivsW \ProjC_{\IR^n}$.
\end{enumerate}
\end{theorem}
\begin{proof}
The inverse reductions of Fact \ref{kridc} can be obtained as follows.

We first deal with the positive representation. According to Remark
\ref{rem:psi+}, let $A:=\overline{\set{a_n}{n\in\IN}}\in \AA_+(\IR^n)$ and $x
\in X$. By \cite[Lemma 5.1.7]{Wei00} we can compute $d(x,A)$ as an element of
$\IR_>$, hence we can determine a (natural) $M>d(x,A)$. Given $x$ we can also
determine an upper bound $N\in\IN$ for $d(0,x)$. Let $K:=\overline{A\cap
B(0,M+N)}$. Using Remark \ref{compeucl}, and since clearly $K \subseteq
\overline B(0,M+N)$, it suffices to compute a dense sequence in $K$. This is
not difficult, starting from the positive information for $A$: we list all
points in $\set{a_n}{n\in\IN}$ with distance from $0$ strictly less than
$M+N$. Notice that all projection points of $x$ onto $A$ belong to $K$.
Obviously, these points also are projections points of $x$ onto $K$, so that
an application of $\ProjK^+_{\IR^n}$ to this new set releases a correct
result.\smallskip

We now deal with the total representation. Let then $A\in\AA(\IR^n)$ be
given. We want to compute, as a suitable input for $\ProjK_{\IR^n}$, some
compact $L$ with total information such that the projections points of $x$
onto $L$ should be also projection points of $x$ onto $A$. However, we cannot
set $L:=K$ (with $K$ the same of the previous case). This is because the
possible elements in $A \cap\partial B(0,M+N)$ that are not accumulation
points of the dense set enumerated in $K$ do not belong to $K$, but they are
inevitably preserved by the negative information on $A$ and on
$\IR^n\setminus\overline B(0,M+N)$. Thus, the two descriptions needed to
provide the total information of $K$ can fail to be coherent. To obtain
consistent information for both types of information, we add to $K$ the whole
set $\partial \overline B(0,M+N)$. Therefore we define $L$ to be
$$
\overline{A \cap B(0,M+N)} \cup \partial B(0,M+N) =
(A \cap \overline B(0,M+N)) \cup \partial B(0,M+N).
$$
The left hand term of the equation guarantees that a $\psi^+_{\IR^n}$-name of
$L$ can be effectively obtained, while the right hand side guarantees the
same with respect to a $\psi^-_{\IR^n}$- name. Finally, use Remark
\ref{compeucl} and the fact that $L\sbsq \overline B(0,M+N)$ to obtain a
$\kappa_{\IR^n}$-name of $L$ as a member of $\KK(\IR^n)$.\smallskip

We now consider the negative representation with the goal of showing that
$\ProjC^-_{\IR^n} \leqW \ProjK^-_{\IR^n}$. We make use of the homeomorphism
$f$ between $\IR^n$ and the open ball $B(0,\frac\pi2)$ defined by
$$
f(t)=
\begin{cases}
\arctan(d(t,0)) \frac t{d(t,0)} & \text{if $t \neq 0$;}\\
0 & \text{if $t=0$.}
\end{cases}
$$
We claim that $f$ is computable. The critical points are the vectors $t$
close to 0, but we can handle them as follows: until the test
$\arctan(d(t,0))>0$ fails and the parallel test $\arctan(d(t,0))<2^{-i}$
succeeds, we let $f(t)[i]=0$. Notice in fact, that for all $t\in\IR^n$
(including 0), $d(f(t),0)= \arctan(d(t,0))$. Analogously, one can prove that
$f^{-1}$ is also computable.

Now suppose we are given $(x,A) \in \dom(\ProjC^-_{\IR^n})$. We compute a
compact set $H \in \KK_-(\IR^n)$ as follows. The main idea is to use $f$ to
rescale $A$ within the compact $\overline B(0,\frac\pi2)$. However, the
function $f$ produces unavoidable metric distortions, as $f$-images get
closer to each other the more the original points are far from the origin.
Hence, projections points of $f(x)$ onto $f(A)$ do not necessarily correspond
to $f$-images of the projection points of $x$ onto $A$. To solve this
problem, we first translate the space so that $x$ becomes the origin: in this
way, the order relationships between distances from $x$ are preserved by $f$.
To take into account that possible \lq\lq infinity points\rq\rq\ of $A$ get
mapped to points on $\partial B(0,\frac\pi2)$ we add the whole $\partial
B(0,\frac\pi2)$ to our compact set. Therefore we set
\begin{align*}
 H & := f(A-x) \cup \partial B\Big(0,\frac\pi2\Big) \\
 & = \set{y \in \IR^n}{d(y,0) \leq \frac\pi2 \land \left(d(y,0) < \frac\pi2 \To f^{-1}(y)+x\in A\right)}.
\end{align*}
The second line provides a $\Pi^0_1$ definition of $H$ with $A$ and $x$ as
parameters, and hence (a name for) $H \in \AA_-(\IR^n)$ is computed from (a
name for) $x$ and $A \in \AA_-(\IR^n)$. Since $H \subseteq \overline B(0,\frac\pi2)$,
by Remark \ref{compeucl} we have $H \in\KK_-(\IR^n)$.

Since $d(t,x)= d(t-x,0)$ and $d(f(t-x),0) = \arctan(d(t-x,0))$, the
monotonicity of $\arctan$ implies that $d(f(t-x),0) \leq d(f(t'-x),0)$ if and
only if $d(t,x) \leq d(t',x)$ for all $t,t' \in \IR^n$. Thus the members of
$\ProjK^-_{\IR^n}(0,H)$ are exactly those of the form $f(t-x)$ for some $t
\in \ProjC^-_{\IR^n}(x,A)$. Therefore from $y \in \ProjK^-_{\IR^n}(0,H)$ we
can compute $f^{-1}(y)+x \in \ProjC^-_{\IR^n}(x,A)$. Notice that we are using
$x$ (which is part of the original input) in this final computation, so that
we do not prove $\ProjC^-_{\IR^n} \leqsW \ProjK^-_{\IR^n}$.
\end{proof}

Since we are interested mainly in projections in Euclidean spaces, Theorem
\ref{CequivK} allows us to concentrate on operators for closed sets.

The proof of the next theorem shows that we can obtain upper bounds for all
 three exact projection operators by using essentially the same argument.

\begin{theorem}\label{wlcupperbound}
\begin{enumerate}
\item For $n \geq 1$, $\ProjC^-_{\IR^n}$ and $\ProjC^+_{\IR^n}$ are non
    deterministically limit computable, that is
    $\ProjC^-_{\IR^n},\ProjC^+_{\IR^n}\leqsW\BWT_\IR$.
\item For $n \geq 1$, $\ProjC_{\IR^n}$ is non deterministically computable,
    that is
$$\ProjC_{\IR^n}\leqsW\WKL.$$
\end{enumerate}
\end{theorem}
\begin{proof}
We first show (2). Given $x\in\IR^n$ and $A\in\mathcal A(\IR^n)$ we can
compute $d(x,A)\in\IR$ by \cite[Lemma 5.1.7]{Wei00}. We use this distance to
compute first $C:=\partial B(x,d(x,A))$ as an element in $\AA_-(\IR^n)$, and
then $A\cap C\in\AA_-(\IR^n)$. This set obviously consists precisely of all
projection points of $x$ onto $A$.

We use then an upper bound $N$ of $d(x,0)$ and an upper bound $M$ of $d(x,A)$
to translate $A\cap C$ into an element of $\KK_-(\IR^n)$: it holds in fact that
$A\cap C\sbsq C\sbsq\overline B(0,N+M)$.

Finally, to determine a projection point of $x$ onto $A$, it suffices to
select a point from this compact set. This is the only non computable step in
the construction, but it is non deterministically computable by \cite[Theorem
2.10]{BG11}. This shows $\ProjC_{\IR^n} \leqW \WKL$, and $\ProjC_{\IR^n}
\leqsW \WKL$ follows because $\WKL$ is a cylinder.\smallskip

When $A$ is not provided with total information, we can initially use $\lim$
to obtain the total information about $A$. For the input $A\in\AA_+(X)$, this
follows by \cite[Proposition 4.2]{BG09}. For the input $A\in\AA_-(X)$ this
follows by \cite[Proposition 4.5]{BG09} (since $\IR^n$ is effectively locally
compact). The remainder of the process remains unaltered. Therefore, the
negative and the positive projection operators can be simulated by composing
a limit computable procedure with a non deterministically computable one,
i.e., they are non deterministically limit computable.

By \cite[Corollary 11.19]{BW} this means that $\ProjC^-_{\IR^n},
\ProjC^+_{\IR^n} \leqW \BWT_\IR$ and again we obtain $\ProjC^-_{\IR^n},
\ProjC^+_{\IR^n} \leqsW \BWT_\IR$ because $\BWT_\IR$ is a cylinder (see
\cite[Corollary 11.13]{BW}).
\end{proof}

\subsection{Exact negative projection operators}

In the previous section we have seen that $\ProjC^-_{\IR^n}\leqsW\BWT_\IR$.
But is this reduction in fact an equivalence? This is indeed the case for $n
\geq 2$ as the following result shows:

\begin{theorem}\label{BWTRK-}
$\BWT_\IR \leqsW \ProjC^-_{\IR^n}$ for $n \geq 2$.
\end{theorem}
\begin{proof}
Recall that by \cite[Corollary 11.7]{BW}, $\BWT_\IR \equivsW \WKL'$. Hence we
can substitute in the proof $\BWT_\IR$ by $\WKL'$. Moreover, it suffices to
work with $n=2$ because the results for $n>2$ follow by transitivity of
$\leqsW$ as $\ProjC^-_{\IR^2} \leqsW \ProjC^-_{\IR^n}$.

Fix the usual (and computable, by \cite[Lemma 7.1]{BGH15}) embedding $\iota:
\Can \to [0,1]$. Moreover, let $h: \IN^{<\IN} \to [1,2]$ be computable and
such that $w <_{lex} w'$ iff $h(w) < h(w')$ for every $w,w' \in \IN^{<\IN}$.

Throughout this proof it is convenient to represent points of $\IR^2$ in
polar coordinates (which we will write $(r, \alpha)$). This does not cause
any problem, because we use only points with radial coordinate not smaller
than $1$ and angular coordinate in the interval $[0,1]$: for such points both
directions of the conversion between Cartesian and polar coordinates are
computable.

Without loss of generality, we are given as input a sequence of trees
$(T_n)_n$ converging to an infinite binary tree $T$ and we want to find,
using $\Proj^-_{\IR^2}$, an infinite path in $T$. To achieve this goal we
compute a closed subset $B$ of $A = \set{(r,\alpha)}{1 \leq r \leq 2 \land
\alpha \in \iota(\Can)} \in \AA^-(\IR^2)$ such that if $(r, \alpha) \in
\ProjC^-_{\IR^n}(0,B)$, then $\iota^{-1}(\alpha)$ is an infinite path in $T$.
$B$ is defined as the intersection of closed sets $B_s$.

To describe the $B_s$, for each $w \in 2^{<\IN}$ and $r_0 \geq 1$, let us
denote by $B_{w,r_0}$ the closed set $A \setminus \set{(r,\alpha)}{r < r_0
\land w \prec \iota^{-1}(\alpha)} \in \AA^-(\IR^2)$ (here, as usual, $w \prec
p$ denotes that the finite binary string $w$ is an initial segment of $p \in
\Can$). In words, $B_{w,r_0}$ is obtained by removing from $A$ the inner
slice up to $r_0$ in the $w$-direction.

At stage $s$, for every $k \leq s$ we let $t_k(s)$ be the cardinality of the
set
\[
\set{n<s}{T_n \cap 2^k \neq T_{n+1} \cap 2^k}.
\]
We then define, for every $k \leq s$,
\begin{align*}
 r_k(s) & = h(\langle t_0(s), \dots, t_{k-1}(s), t_k(s) \rangle), \text{ and}\\
 r_k^+(s) & = h(\langle t_0(s), \dots, t_{k-1}(s), t_k(s)+1 \rangle).
\end{align*}
Eventually, we let
\[
B_s = \bigcap_{k \leq s} \big(\bigcup_{w \in 2^k \cap T_s} B_{w,r_k(s)} \cup \bigcup_{w \in 2^k \setminus T_s} B_{w,r_k^+(s)}\big).
\]

In this way we compute $B = \bigcap_s B_s \in \AA^-(\IR^2)$ and we need to
show that if $(r, \alpha) \in \Proj^-_{\IR^2} (0,B)$  then $\iota^{-1}
(\alpha)$ is a path in $T$.

Since $(T_n)_n$ converges, for every $k$ the sequence $(t_k(s))_s$ is
non-decreasing and eventually takes a constant value $t_k$. Therefore the
sequences $(r_k(s))_s$ and $(r_k^+(s))_s$ stabilize at $r_k = h(\langle t_0,
\dots, t_{k-1}, t_k \rangle)$ and $r_k^+ = h(\langle t_0, \dots, t_{k-1}, t_k
+1\rangle)$ respectively.

If $p \in [T]$ then the ray starting at $0$ and moving in direction
$\iota(p)$ meets $B$ at distance $\sup \set{r_k}{k \in \IN}$ from $0$. To see
this notice first that if $t_k(s')>t_k(s)$ then $r_k(s') = h(\langle t_0(s'),
\dots, t_{k-1}(s'), t_k(s') \rangle) \geq h(\langle t_0(s), \dots,
t_{k-1}(s), t_k(s) +1 \rangle) = r_k^+(s)$. Thus, even when for some $s \geq
k$ with $p \restriction k \notin T_s$ we deleted the ray in direction
$\iota(p)$ up to $r_k^+(s)$, at some later stage $s'$ (such that $p
\restriction k \in T_{s'}$ and so $t_k(s')>t_k(s)$) the deletion up to
$r_k(s') \leq r_k$ superseded it.

If instead $p \notin [T]$ and $\ell$ is least such that $p \restriction \ell
\notin T$ then the ray starting at $0$ and moving in direction $\iota(p)$
meets $B$ at distance $\geq r_\ell^+$ from $0$ (because at a stage $s$ such
that $T_s \cap 2^{\leq \ell} = T \cap 2^{\leq \ell}$ we delete the ray up
to distance $r_\ell^+$).

It thus suffices to check that $\sup \set{r_k}{k \in \IN} < r_\ell^+$ for
every $\ell$. Indeed we have
\begin{multline*}
\sup \set{r_k}{k \in \IN} = \sup \set{h(\langle t_0, \dots, t_k \rangle)}{k \in \IN} \leq \\
h(\langle t_0, \dots, t_\ell, t_{\ell+1} +1 \rangle) < h(\langle t_0, \dots, t_\ell +1 \rangle) = r_\ell^+.
\end{multline*}
Thus every point in $\Proj^-_{\IR^2} (0,B)$ is in direction $\iota(p)$ for
some $p \in [T]$, as required.
\end{proof}

\begin{corollary}\label{equiv-BWT}
$\ProjC^-_{\IR^n}\equivsW\BWT_\IR$ for $n \geq 2$.
\end{corollary}
\begin{proof}
By Theorem \ref{wlcupperbound}.(1) and Theorem \ref{BWTRK-}.
\end{proof}

For case $n=1$ we do not obtain the full power of $\BWT_\IR$. We can prove
however that a characterization for the one dimensional case can be found in
terms of $\BWT_2$. As a preliminary result, we prove:

\begin{proposition}\label{negativeIR}
$\ProjC^-_\IR\leqW\LLPO*\lim$.
\end{proposition}
\begin{proof}
Analogously to the treatment of negative information in the proof of Theorem
\ref{CequivK}, given $x\in\IR$ and $A\in\AA_-(\IR)$ we can compute the set
$$B:=\arctan(A-x)\cup\{-\frac\pi2\}\cup\{\frac\pi2\}\in\AA_-(\IR).$$ Notice
that for such a set both
\begin{center}
$\ell:=\max\set{y\in B}{y\leq 0}$ and
$r:=\min\set{y\in B}{y\geq 0}$
\end{center}
always exist, which would not hold true in general for our original $A$.
Moreover $\ProjC^-_\IR(0,B) \sbsq \{\ell,r\}$. Since $A\neq\eps$,
$|y|<\frac\pi2$ for all $y\in\ProjC^-_\IR(0,B)$. More precisely, as in the
proof of Theorem \ref{CequivK}, the members of $\ProjC^-_\IR(0,B)$ are
exactly those of the form $\arctan(t-x)$ for some $t\in\ProjC^-_\IR(x,A)$.
Recalling that $B\in\AA_-(\IR)$ we can determine $\ell$ as an element of
$\IR_>$ and $r$ as an element of $\IR_<$. We then use $\lim \times \lim
\equivsW \lim$ to obtain $\ell,r\in\IR$. Let now denote as $f_0$ the function
mapping any given  $B\in\AA_-(\IR)$ for which the elements $\ell,r$ defined
as above exist to the pair $(\ell,r)\in\IR^2$. We have then just proved that
$f_0\leqsW\lim$.

Consider now the function $g_0:\IR^2\to\IR$ such that $g_0(z_0,z_1) \in
\{z_0,z_1\}$ and $|g_0(z_0,z_1)| = \min\{|z_0|, |z_1|\}$. It is easy to see
that $g_0 \leqW \LLPO$ (an application of $\LLPO$ finds $i<2$ such that
$|z_i| \leq |z_{1-i}|$, then we use the input $(z_0,z_1)$ of $g_0$, which is still available by definition of $\leqW$, to recover the value of $z_i$).

Let now $y:=g_0(\ell,r)$. Then, in virtue of what observed above,
$\tan(y)+x\in\ProjC^-_\IR(x,A)$.

This shows that $\ProjC^-_\IR\leqW g_0\circ f_0$ (the transformation $(x,A)
\mapsto B$ was indeed computable uniformly in $(x,A)$ and notice also that,
by definition of $\leqW$, the original $x\in\IR$ is still available after the
application of $g_0\circ f_0$, hence it can be used to compute $\tan(y)+x$).

By definition of compositional product, $g_0 \circ f_0 \leqW \LLPO*\lim$ for
every $g_0 \leqW \LLPO$ and $f_0 \leqW \lim$. Therefore $\ProjC^-_\IR \leqW
\LLPO*\lim$.
\end{proof}

For the next result we need to use the Sierpinski space and its ordinary
admissible representation:

\begin{definition}[Sierpinski space]
The \emph{Sierpinski space} is given by the topology
$\Sie:=\{\{1\},\{0,1\}\}$  on the set $2:=\{0,1\}$.

As a represented space, the Sierpinski space is equipped with the
representation $\delta_\Sie(0^\IN)=0$ and $\delta_\Sie(p)=1$ for
$p\neq0^\IN$.
\end{definition}

In other words, $\LPO$ can be seen as the identity function
$\id_{\Sie,2}:\Sie\to\{0,1\}, i\mapsto i$, where the codomain is equipped
with the discrete topology.


\begin{lemma}\label{BWT2limneg}
$\BWT_2\times\lim\leqsW\ProjC^-_\IR$.
\end{lemma}
\begin{proof}
Consider the space $\Sie^\IN$. As $\id_{\Sie,2}\equivsW\LPO$, for the
identity function $\id_{\Sie^\IN,\Can}:\Sie^\IN\to\Can$ we find that
$\id_{\Sie^\IN,\Can}\equivsW\lim$, as obviously
$\id_{\Sie^\IN,\Can}\equivsW\widehat{\id_{\Sie,2}}$ and moreover
$\lim\equivsW\widehat\LPO$ (\cite[Lemma 6.3]{BG11a}, \cite[Theorem
6.7]{Weihrauch}). In the statement we can thus replace $\lim$ with
$\id_{\Sie^\IN,\Can}$.


Notice now that the computable embedding $\iota: \Can \to [0,1]$ we already
used in the proof of Theorem \ref{BWTRK-} gives naturally rise to a
corresponding computable embedding $\iota_\Sie:\Sie^\IN\to\IR_<$. Recall that
$\iota$ preserves the order on binary sequences, hence $\iota_\Sie$ does the
same.

Finally, observe that $\Sort_\Sie:\Can\to\Sie^\IN$, the lifted version of
$\Sort$ such that $\Sort_\Sie(p)\in\Sie^\IN$ coincides with
$\Sort(p)\in2^\IN$, is computable.

In the following, by notational abuse, we identify $\Sie^\IN$ with
$\dom(\delta_{\Sie^\IN})=\Bai$, that is, we will not distinguish a binary
sequence $(i_0,i_1,i_2,\dots)\in\Sie^\IN$ from any $\langle p_0, p_1, p_2,
\dots \rangle\in\Bai$ such that $\delta_{\Sie^\IN}(\langle p_0, p_1, p_2,
\dots \rangle) = (\delta_\Sie(p_0), \delta_\Sie(p_1), \delta_\Sie(p_2),
\dots) = (i_0, i_1, i_2, \dots)$. This produces no ambiguity for
$\id_{\Sie^\IN,\Can}$ and $\iota_\Sie$ that still remain single-valued,
whereas the single-valuedness of $\Sort_\Sie$ can be preserved by its
replacement with a computable realizer. For instance, we will see
$\iota_\Sie$ as defined by $\iota_\Sie(\langle p_0, p_1, p_2, \dots\rangle)
:= \iota((\delta_\Sie(p_0), \delta_\Sie(p_1), \delta_\Sie(p_2), \dots)$, and
$\Sort_\Sie(p)$ as being of the form $\langle p_0, p_1, p_2, \dots\rangle$
with $\Sort(p)=(\delta_\Sie(p_0), \delta_\Sie(p_1), \delta_\Sie(p_2),
\dots)$.

Now, given inputs $(p,q)\in\BWT_2\times\id_{\Sie^\IN,\Can}$, we compute
$\ell:=-\iota_\Sie(\langle\Sort_{\Sie}(p),q\rangle)\in\IR_>$ and
$r:=\iota_\Sie(\langle\Sort_{\Sie}(1-p),q\rangle)\in\IR_<$, where
$$\langle\Sort_{\Sie}(p),q\rangle:=\langle p_0, q_0, p_1, q_1, p_2, q_2,
\dots\rangle$$
for $\Sort_\Sie(p):=\langle p_0, p_1, p_2, \dots\rangle$ and
$q:=\langle q_0, q_1, q_2,\dots\rangle$. Since $\iota_\Sie(\Sort_\Sie(s))$
coincides with $\iota(\Sort(s))$ for all $s\in\Bai$, we notice that
$|\ell|\leq|r|$ if and only if only if $p$ contains infinitely many $0$ (in
this case indeed $\Sort(p)=0^\IN$), and $|\ell|\geq|r|$ if and only if $p$
contains infinitely many $1$ (in this case indeed $\Sort(1-p)=0^\IN$).

Given now $\ell\in\IR_>$ and $r\in\IR_<$, we can compute
$$A:=\set{x\in\IR}{x\leq \ell-1\vee x\geq r+1}\in\AA_-(\IR).$$
From $y\in\ProjC^-_\IR(0,A)\sbsq\{\ell-1,r+1\}$ we can use
$\iota^{-1}:\sbsq\IR\to\Can$ to compute $\langle\Sort_\Sie(p),q\rangle$ and
then $q\in 2^\IN$. Moreover, the sign of $y$ yields a valid answer to
$\BWT_2(p)$ (notice that the sign of $\ell-1, r-1$ is always decidable, as
they are necessarily different from $0$, which is not necessarily the case
for $\ell$ and $r$).
\end{proof}

Through the notion of jump that we recalled in Section \ref{ssec:milestones}
we are now able to characterize $\ProjC^-_\IR$ in terms of $\BWT_2$:

\begin{corollary}\label{BWT2Xlimneg}
$\ProjC^-_\IR\equivW\BWT_2\times\lim$.
\end{corollary}
\begin{proof}
This follows by Proposition \ref{negativeIR} and Lemma \ref{BWT2limneg},
since $\BWT_2\equivsW\LLPO'$ (\cite[Corollary 11.11]{BW}) and since for
generic multi-valued functions $f$ it holds $f*\lim\equivW f'\times\lim$.
\end{proof}

\subsection{Exact positive projection operators}

Quite surprisingly, for the projections with positive information for closed
sets we obtain the same characterizations obtained for the case of negative
information. We start with the dimensions $n \geq 2$ for which we are still
able to prove the equivalence with $\BWT_\IR$:

\begin{proposition}\label{BWTRK+}
$\BWT_\IR \leqsW \ProjC^+_{\IR^n}$ for $n\geq 2$.
\end{proposition}
\begin{proof}
We prove the statement for $n=2$. As before, the results for $n>2$ follow by
transitivity of $\leqsW$ as $\ProjC^+_{\IR^2}\leqsW\ProjC^+_{\IR^n}$. As in
the proof of Theorem \ref{BWTRK-}, also here it is convenient to represent
points of $\IR^2$ in polar coordinates. In this case we use only points with
radial coordinate not smaller than $1$ and angular coordinate in the interval
$[-\frac\pi2, \frac\pi2]$, so that again for our purposes both directions of
the conversion between Cartesian and polar coordinates are computable.

Let $(a_n)_n\in\dom(\BWT_\IR)$ be given as input. We want to find a cluster
point of this sequence. We consider the points $b_n = (1+2^{-n},
\arctan(a_n))$. Let now $A: = \overline{\set{b_n}{n\in\IN}} \in
\AA_+(\IR^2)$. Notice that $(1, \pm\frac\pi2) \notin A$ because $(a_n)_n$ is
bounded, while $(1,\alpha) \in \ProjC^+_{\IR^2}(0,A)$ if and only if
$\tan(\alpha)$ is a cluster point of $(a_n)_n$. Thus if $(r,\alpha) \in
\ProjC^+_{\IR^2} (0,A)$ then $r=1$ and $\tan(\alpha) \in \BWT_\IR ((a_n)_n)$.
\end{proof}

\begin{corollary}\label{equiv+BWT}
$\ProjC^+_{\IR^n} \equivsW \BWT_\IR$ for $n\geq 2$.
\end{corollary}
\begin{proof}
By Theorem \ref{wlcupperbound}.(1) and Proposition \ref{BWTRK+}.
\end{proof}

For $n=1$, by reasoning analogously to the case of negative information, we obtain the same characterization in terms of $\BWT_2$. We start with the following result, which is an analoguous of Proposition \ref{negativeIR}:

\begin{proposition}\label{positiveIR}
$\ProjC^+_\IR\leq_W\LLPO*\lim$.
\end{proposition}
\begin{proof}
Let $A:=\set{x_n}{n\in\IN}\in\AA_+(\IR)$ and $x\in\IR$ be given. We can now compute
$$B:=A\cup\{x-d(x,x_0)-1\}\cup\{x+d(x,x_0)+1\}\in\AA_+(\IR).$$
Notice that both $\ell:=\max\set{y\in B}{y\leq x}$ and $r:=\min\set{y\in
B}{y\geq x}$ always exist, which would not hold true in general for our
original $A$. Moreover $\ProjC^+_\IR(x,B)=\ProjC^+_\IR(x,A)$. Recalling that
$B\in\AA_+(\IR)$ we can determine $\ell$ as an element of $\IR_<$ and $r$ as
an element of $\IR_>$. Analogously to the proof of Proposition
\ref{negativeIR}, we can then use $\lim\times\lim\equivsW\lim$ to obtain
$\ell,r\in\IR$ and consequently $\LLPO$ to determine $y \in \{\ell,r\}$ such
that $d(y,x) = \min\{d(\ell,x),d(r,x)\}$. As observed, this gives a member of
$\ProjC^+_{\IR^n}(x,A)$.
\end{proof}

\begin{lemma}\label{BWT2limpos}
$\BWT_2\times\lim\leqsW\ProjC^+_\IR$.
\end{lemma}
\begin{proof}
The proof is almost the same of that of Lemma \ref{BWT2limneg}. But the
replacement of the negative representation for closed sets with its dual
requires to switch the positions of $\ell$ and $r$ with respect to $0$. We
compute then the new set $A:=\set{x\in\IR}{x\leq r-2\vee x\geq
\ell+2}\in\AA_+(\IR)$. From $y\in\ProjC^+(0,B)$ we can then extract $q$.
Moreover, given the sign of $y$ , we will select 0 or 1 as accumulation point
of the input sequence by making the dual choice with respect to that of the
proof of Lemma \ref{BWT2limneg}.
\end{proof}

\begin{corollary}\label{BWT2Xlimpos}
$\ProjC^+_\IR\equivW\BWT_2\times\lim$.
\end{corollary}
\begin{proof}
This follows by Proposition \ref{positiveIR} and Lemma \ref{BWT2limpos}, analogously to the proof of Corollary \ref{BWT2Xlimneg}.
\end{proof}

\subsection{Exact total projection operators}

For the case of total information we can fully characterize the Weihrauch
degree of $\ProjC_{\IR^n}$ already for $n=1$. We start by determining the
following upper bound:

\begin{proposition}\label{ProjCLLPO}
$\ProjC_\IR\leqW \LLPO$.
\end{proposition}
\begin{proof}
Let the input $(x,A)$ be given with $A\in\AA(\IR)$. It holds obviously that
$0<|\ProjC_\IR(x,A)|\leq 2$, and in fact $\ProjC_\IR(x,A)=\{x-r,x+r\}\cap A$ for
$r:=d(x,A)$. By using the total information on $A$ we can compute the exact
value of $r$ via approximations that become at every stage more reliable. We
produce then a valid input $\langle p_0,p_1 \rangle$ for $\LLPO$ in the following
way. At stage $s$, by considering the initial segment of the negative
information on $A$ that we have read so far, if both points $x-r$ and $x+r$
still are plausible candidates as members of $A$ we let $p_0(s):=0=:p_1(s)$.
Otherwise, suppose that we realize that one of the two points, say $x-r$, is
not in $A$. Then we put $p_0(s):=1 \neq 0=:p_1(s)$. We then let $p_0(s+m):=0=:p_1(s+m)$
for all $m>0$. If instead we realize that $x+r \notin A$ then we switch the
roles of $p_0$ and $p_1$.

Given now $i\in\LLPO(\langle p,q \rangle)$ we compute (again)
$x-r$ or $x+r$, depending on whether $i=0$ or $i=1$, finding an element of $\Proj_\IR(A,x)$.
\end{proof}

Notice that in the above proof the use of the original input after the
application of $\LLPO$ is essential: $\ProjC_\IR \leqsW \LLPO$ cannot hold
for mere cardinality reasons. But the opposite reduction even holds for the
strong version of Weihrauch reducibility:

\begin{proposition}\label{LLPOProjC}
$\LLPO\leqsW\ProjC_\IR$.
\end{proposition}
\begin{proof}
Let $\langle p_0,p_1\rangle\in\dom(\LLPO)$. We construct then a valid input
$(x,A)$ for $\ProjC_\IR$ according to the following idea: if a point of
$\ProjC_\IR(0,A)$ is negative, then $p_0=0^\IN$, and if a point of
$\ProjC_\IR(0,A)$ is positive, then $p_1=0^\IN$. If we do this then by
checking the sign of an element of $\ProjC_\IR(0,A)$ we determine an element
of $\LLPO(\langle p_0,p_1\rangle)$.

The construction of $A$ proceeds as follows. We immediately remove from $A$
the intervals $]-1,1[$, $]-\infty,-2[$ and $]2,\infty[$. Then we activate the
following inductive procedure. Suppose that at stage $s\geq 0$ it holds
$p_0(m)=0=p_1(m)$ for all $m\leq s$. We add then to $A$ the points
$x_{0,s}:=-1-2^{-s}$ and $x_{1,s}:=1+2^{-s}$. At the same time we remove from
$A$ the intervals $]x_{0,s-1},x_{0,s}[$ and $]x_{1,s},x_{1,s-1}[$. Otherwise,
let $s$ be the first stage in which a digit different from 0 appears in
$\langle p_0,p_1 \rangle$, say, $p_0(s)\neq 0$. We want then the closest
point of $A$ to 0 to be positive. To this aim, we add to $A$ the point
$x_{1,s}:=1+2^{-s}$ alone. Moreover, we remove from $A$ the intervals
$]x_{0,s-1},0[$, $]0,x_{1,s}[$ (notice that 0 was removed from $A$ already
before the start of the inductive procedure), and $]x_{1,s},x_{1,s-1}[$. In
this case the description of $A$ is complete after stage $s$. The case
$p_1(s)\neq 0$ is analogous.
\end{proof}

\begin{corollary}\label{ProjCequivLLPO}
$\ProjC_\IR\equivW\LLPO$.
\end{corollary}
\begin{proof}
By Propositions \ref{ProjCLLPO} and \ref{LLPOProjC}.
\end{proof}

For $n \geq 2$ we see instead that a precise characterisation is given by
$\WKL$.

\begin{theorem}\label{WKLProjC}
$\WKL\leqsW\ProjC_{\IR^n}$ for $n\geq 2$.
\end{theorem}
\begin{proof}
We prove the statement for $n=2$ by replacing $\WKL$ through its well known
strongly Weihrauch equivalent version $C_{[0,1]}$ (see \cite[Corollary 4.6]{BdBP}). The cases
$n>2$ are then as usual proved by transitivity of $\leqsW$.

Let then $A \in \AA_-([0,1])$ be given, which means that we are provided with
a sequence of rational open intervals $(I_n)_n$ such that $[0,1] \setminus A
= \bigcup_{n\in\IN} I_n$. We now construct the new closed set $K \in
\AA(\IR^2)$ as the set of all points (with polar coordinates, as in the
proofs of Theorem \ref{BWTRK-} and Proposition \ref{BWTRK+}) $(r, \alpha)$
satisfying the following three conditions:
\begin{enumerate}
\item $1 \leq r \leq 2$,
\item $0 \leq \alpha \leq 1$,
\item $(\forall n)(\alpha \in I_n \To 1+2^{-n+1} \leq r)$.
\end{enumerate}
Intuitively, we draw in $\IR^2$ part of the circular crown between the
circles of radius $1$ and $2$ centered at the origin. We then remove around a
point $(1,\alpha)$ a little open portion of the crown as soon as we know that
$\alpha \notin A$ (see Figure \ref{fig:WKLProjC}).
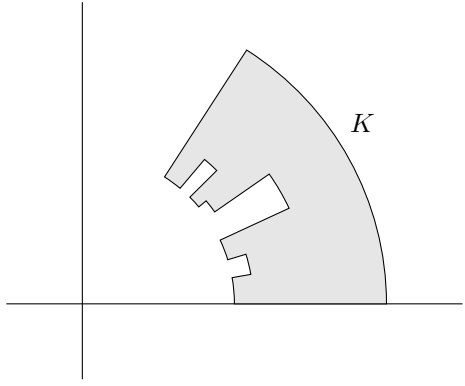
\begin{figure}
\begin{tikzpicture}[scale=2]
 \draw (-.5,0) -- (2.5,0);
 \draw (0,-.5) -- (0,2);
 \filldraw[fill=gray!20, draw=black] (57.29:2) arc [radius=2cm, start angle=57.29, end angle=0] -- (0:1)
 arc [radius=1cm, start angle=0, end angle=10] -- (10:1.125) arc [radius=1.125cm, start angle=10, end angle=17] -- (17:1)
 arc [radius=1cm, start angle=17, end angle=25] -- (25:1.5) arc [radius=1.5cm, start angle=25, end angle=35] -- (35:1.0625)
 arc [radius=1.0625cm, start angle=35, end angle=40] -- (40:1)
 arc [radius=1cm, start angle=40, end angle=45] -- (45:1.25) arc [radius=1.25cm, start angle=45, end angle=50] -- (50:1)
 arc [radius=1cm, start angle=50, end angle=57.29] -- cycle;
 \node [right] at (1.7,1.2) {$K$};
\end{tikzpicture}
\caption{The construction of the set $K\subseteq\mathbb R^2$ in the proof of
Theorem \ref{WKLProjC}: here $I_0,I_3\subseteq[0,1]\setminus A$ overlap.
\label{fig:WKLProjC}}
\end{figure}

It is immediate to see that if $(r,\alpha) \in \ProjC_{\IR^2}(0,K)$ then
$r=1$ and $\alpha \in A$. Hence it remains to prove that we can compute a
name of $K \in \AA(\IR^2)$.

To see that (a name for) $K$ as an element of $\AA_-(\IR^2)$ is computable
from (the given name for) $A$, observe that all the conditions (1)--(3) are
$\Pi^0_1$ in $(I_n)_n$.

To see that also (a name for) $K$ as an element of $\AA_+(\IR^2)$ is
computable from (the given name for) $A$, observe that $K$ is the closure of
the set
$$
C:=\set{(r,\alpha) \in \IQ \times \IQ}{(r,\alpha) \in K \land r>1}.
$$
We claim that we can enumerate, and even decide, this subset of $\IQ \times
\IQ$ effectively from $(I_n)_n$. The conditions $1<r\leq2$ and
$0\leq\alpha\leq1$ are immediately decidable for rational numbers. Hence, to
determine whether $(r,\alpha)\in K$ it remains only to analyze the condition
(3) in the definition of $K$. To this aim, for $1<r\leq 2$, let then
$m\in\IN$ be minimal such that $1+2^{-m+1}\leq r$. Then, for $1\leq\alpha\leq
2$ condition (3) is equivalent to $\alpha \notin \bigcup_{n<m} I_n$. Since
$r$ and $\alpha$ are rational numbers, we can find effectively such $m$ and
then decide whether $\alpha \in \bigcup_{n<m}I_n$.

Therefore the set $C$ is decidable and we can enumerate its members (as pairs
of real numbers) for the positive information on $K$.
\end{proof}

\begin{corollary}\label{ProjCequivWKL}
$\ProjC_{\IR^n} \equivsW \WKL$ for $n \geq 2$.
\end{corollary}
\begin{proof}
By Theorem \ref{wlcupperbound}.(2) and Theorem \ref{WKLProjC}
\end{proof}

\section{Approximate projections}\label{sec:approx}

Since, as we have seen, the (exact) projection operators are computationally
quite hard, it might be reasonable to consider some approximate versions of
them. In many practical circumstances, we may indeed be content of finding
points that lie at a distance comparable with the smallest one.

\begin{definition}
Given a metric space $X$, $\var>0$, a point $x \in X$ and a nonempty set $A
\subseteq X$ we say that $y \in A$ is a \emph{$\var$-projection point of $x$
onto $A$} if $d(x,y) \leq (1+\var) \: d(x,A)$. In other words, the
$\var$-projection points of $x$ onto $A$ are the points of $A$ which are at
minimal distance from $x$ up to an error of $\var$ times the distance itself.
\end{definition}

Notice that if $x \in A$ then for any $\var$, $x$ is the unique
$\var$-projection point of $x$ onto $A$. In general, for any $\var$,
$\var$-projection points of $x$ onto $A$ exist unless $x \in \bar{A}
\setminus A$. As when dealing with exact projections, we will be interested
in the case where $A$ is closed; in this situation $\var$-projection points
of any $x \in X$ onto $A$ do exist for any $\var$. If $X$ is a computable
metric space, the multi-valued functions arising from $\var$-projections
points and depending on the representation of $A \subseteq X$ are defined
similarly to their exact counterparts.

\begin{definition}\label{approx}
Given a computable metric space $X$ and $\var>0$ the \emph{$\var$-approximate
negative, positive and total closed projection operators on $X$} are the
partial multi-valued functions $\ProjvarC^-_X$, $\ProjvarC^+_X$ and
$\ProjvarC_X$ which associate to every $x \in X$ (with Cauchy representation)
and every closed $A \neq \emptyset$ (with negative, positive and total representation,
respectively) the set of the $\var$-projection points of $x$ onto $A$.

Thus $\ProjvarC^-_X: \sbsq X \times \AA_-(X) \toto X$, $\ProjvarC^+_X: \sbsq
X \times \AA_+(X) \toto X$, and $\ProjvarC_X: \sbsq X \times \AA(X) \toto X$.

The \emph{$\var$-approximate negative, positive and total projections
operators for compact sets} are defined by replacing $\AA_-(X)$, $\AA_+(X)$,
and $\AA(X)$ with $\KK_-(X)$, $\KK_+(X)$, and $\KK(X)$ respectively. These
are denoted $\ProjvarK^-_X$, $\ProjvarK^+_X$, and $\ProjvarK_X$.
\end{definition}

The first observations about the approximated operators partly mimic the ones
we made for the exact operators.

\begin{fact}\label{approxkridc}
Let $X$ be a computable metric space and $\var>0$.
\begin{enumerate}
\item If $0<\var'<\var$ then $\var\text{-}\operatorname{P} \leqsW
    \var'\text{-}\operatorname{P} \leqsW \operatorname{P}$ where
    $\operatorname{P}$ is any of $\ProjK^-_X$, $\ProjC^-_X$, $\ProjK^+_X$,
    $\ProjC^+_X$, $\ProjK_X$, $\ProjC_X$.
\item $\K_X \leqsW \ProjvarK^-_X$ and $\C_X \leqsW \ProjvarC^-_X$.
\item $\ProjvarK^-_X \leqsW \ProjvarC^-_X$, $\ProjvarK^+_X \leqsW
    \ProjvarC^+_X$, $\ProjvarK_X \leqsW \ProjvarC_X$.
\item If $X$ is computably compact $\ProjvarK^-_X \equivsW \ProjvarC^-_X$,
    $\ProjvarK^+_X \equivsW \ProjvarC^+_X$, $\ProjvarK_X \equivsW
    \ProjvarC_X$.
\item $\ProjvarK^+_{\IR^n} \equivsW \ProjvarC^+_{\IR^n}$,
    $\ProjvarK_{\IR^n} \equivsW \ProjvarC_{\IR^n}$ for $n \geq 1$.
\end{enumerate}
\end{fact}
\begin{proof}
(1) and (2) are obvious.

(3) and (4) can be proved exactly as Facts \ref{kridc} and \ref{kccompact}
respectively: indeed those proofs consist of transformations of the input and
do not use any specific feature of the functions involved.

(5) follows from the proofs of the analogous results in Theorem
\ref{CequivK}, since the $\var$-projection points of $x$ onto the compact
sets $K$ and $L$ constructed there are also $\var$-projection points of $x$
onto the original closed set $A$.
\end{proof}

Notice that in (5) above $\ProjvarK^-_{\IR^n} \equivsW \ProjvarC^-_{\IR^n}$
is missing -- we show below in Proposition \ref{prop:compactapprox} that this
does not hold. The proof of the analogous result in Theorem \ref{CequivK}
cannot be translated to the approximate setting. Indeed if we repeat that
construction then to obtain the $\var$-projection points of $x$ onto $A$ we
need to have a $\var'$-projection point of $x$ onto $H$ for
$$\var' \leq\frac{\arctan(d(x,A) (1+ \var))}{\arctan(d(x,A))} -1.$$ Hence no specific
$\var'$ will work for all $x$ and $A$. Even viewing $\var$ as part of the
input we do not solve the problem: from the negative information on $A$ we
obtain only lower bounds for $d(x,A)$.

\subsection{Approximated negative projection operators}

The following results characterizes the computational complexity of negative
approximated projection operators on $\IR^n$ for all $n\geq1$.

\begin{theorem}\label{approxequivCR}
For every $\var>0$ and $n \geq 1$, $\ProjvarC^-_{\IR^n} \equivsW \C_\IR$.
\end{theorem}
\begin{proof}
For the right-to-left direction, observe that $\C_\IR \leqsW \C_{\IR^n}
\leqsW \ProjvarC^-_{\IR^n}$ by Fact \ref{approxkridc}.(2).

For the other direction, consider an input $(x,A) \in
\dom(\ProjvarC^-_{\IR^n})$. Since we can compute $d(x,A) \in \IR_<$, we
denote by $r'_s\in\IQ$ the strict lower bound for $d(x,A)$ computed at stage
$s$, so that $\lim_{s \to \infty} r'_s = d(x,A)$. We set
$r_s=\max\{r'_s,0\}$.

We now define the negative closed set
$$
B:=\set{(y,s)}{y\in A \land s\in\IN \land d(x,y)\leq(1+\var)\:r_s}\sbsq\IR^{n+1}.
$$
To see that we can compute a $\psi^-_{\IR^n}$-name of $B$ observe that $B$ is
defined by a $\Pi^0_1$-formula with $A$ as a parameter.

Intuitively, $B$ is constituted by ``copies'' of different subsets of
$A\sbsq\IR^n$ translated onto different levels of the space $\IR^{n+1}$, so
that (i) each copy lies at distance $1$ from the adjacent copies, (ii) on the
$s$-th level we remove the points of $A$ that are ``too far'' from $x$
according to the approximation of $(1+\var)\:d(x,A)$ that we know at that
stage. Notice that the $s$-th level of $B$ is nonempty if and only if $r_s
\geq \frac {d(x,A)} {1+\var}$, and this happens for some $s$ (for all $s$
when $d(x,A)=0$) because $\sup\set{r_s}{s\in\IN}=d(x,A)$. Therefore $B \neq
\emptyset$ is a valid input for $\C_{\IR^{n+1}}$.

If $(y,s) \in \C_{\IR^{n+1}}(B)$, then $y \in \ProjvarC^-(x,A)$ because
$d(x,y) \leq (1+\var)\:r_s \leq (1+\var)\:d(x,A)$. We have shown
$\ProjvarC^-_{\IR^n} \leqsW \C_{\IR^{n+1}}$. Finally
$\C_{\IR^{n+1}}\leqsW\C_\IR$ by \cite[Corollary 4.9]{BdBP}.
\end{proof}

For $\ProjvarK^-_{\IR^n}$ we only state the bounds given in the following
proposition. We recall the use of the finite-parallelization operator that
maps any given multi-valued $f:\sbsq X\toto Y$ to $f^*:\sbsq X^*\toto Y^*$
defined as $f^*(n,x_1,..,x_n):=\{n\}\times f(x_1)\times...\times f(x_n)$ for
all $n\in\IN$.

\begin{proposition}
\label{prop:compactapprox} For every $\var>0$ and $n \geq 1$,
\[
\WKL \leqsW \ProjvarK^-_{\IR^n} \leqW \WKL * \LPO^* * \LPO \lW \C_\IR.
\]
\end{proposition}
\begin{proof}
The first inequality follows from Fact \ref{approxkridc}.(2) and the fact
that $\WKL \equivsW \K_{\IR^n}$ by \cite[Theorem 8.5]{BG11a}.

We proceed to
show that $\ProjvarK^-_{\IR^n} \leqW \WKL * \LPO^* * \LPO$:

Given an input $(x,K) \in \dom(\ProjvarK^-_{\IR^n})$, we can use $\LPO$ to
decide whether or not $x \in K$. If yes, we can output $x$. If no, we can
compute a lower bound $2^{-k} < d(x,K)$. Since we know $K$ as a compact set,
we can subsequently compute some $L \in \mathbb{N}$ such that $K \subseteq
B(x, 2^{-k}(1 + \varepsilon)^{L+1})$. Consider the slices $D_\ell := \{y \in
\IR^n \mid 2^{-k}(1 + \varepsilon)^\ell \leq d(x,y) \leq 2^{-k}(1 +
\varepsilon)^{\ell+1}\}$, which we can effectively compute as closed sets. We
can use $\LPO^L$ to decide for each $\ell < L$ whether $D_\ell \cap K =
\emptyset$ (as we have $K$, and thus also $D_\ell \cap K$ as a compact set).
Let $\ell_0$ be the least positive answer (if it exists), or $L$ otherwise.
Then $D_{\ell_0} \cap K$ is available as a non-empty compact set, and each of
its elements (chosen by $\K_{\IR^n}\equivW\WKL$) is a valid output for
$\ProjvarK^-_{\IR^n}(x,K)$.

That $\WKL * \LPO^* * \LPO \leqW \C_\IR$ is straightforward, e.g.\ via the
independent choice theorem implying the closure under compositional product
of the class of non deterministic functions with finitely many mind changes
(\cite[Theorem 7.6]{BdBP}). To see that this is strict, we observe that
$\C_\IN \nleqW \WKL * \LPO^* * \LPO$. Since the degree of $\C_\IN$ admits a
single-valued representative (for example unique choice \cite{BdBP}), the
closed choice elimination theorem (as stated in \cite[Theorem 2.1]{LP})
implies that if $\C_\IN \leqW \WKL * \LPO^* * \LPO$, then $\C_\IN \leqW
\LPO^* * \LPO$. That this is impossible can be seen using Hertling's level
\cite{hertling}, which is an ordinal invariant of Weihrauch degrees defined
as follows: Given a function $f$, let $D_0 = \dom(f)$, let $D_{\alpha+1}$ be
the closure of the set of discontinuity points of $f|_{D_\alpha}$, and for
limit ordinal $\gamma$, let $D_\gamma = \bigcap_{\beta < \alpha} D_{\beta}$.
The level of $f$ is the least $\alpha$ such that $D_\alpha = \emptyset$ (if
this ever happens). The level of $\C_\IN$ does not exist \cite{paulymaster},
whereas $\LPO^* * \LPO$ has level at most $\omega \cdot 2$.
\end{proof}

Notice that our proof shows in fact that
\[
\ProjvarK^-_X \leqW \WKL * \LPO^* * \LPO,
\]
and hence $\ProjvarK^-_X \lW \C_\IR$, for every computable metric space $X$.

\subsection{Approximated positive projection operators}

\begin{theorem}\label{var+omega1}
For every $\var>0$ and $n \geq 1$, $\ProjvarC^+_{\IR^n} \leqW \Sort$.
\end{theorem}
\begin{proof}
By Proposition \ref{sortomega} it suffices to show that $\ProjvarC^+_{\IR^n}
\leqW \min^-_{\omega+1}$ and by Fact \ref{approxkridc}.1 we can assume that
$\var$ is computable. Given $(x,\overline{\{x_n\}}_n) \in \dom
(\ProjvarC^+_{\IR^n})$ let
$$
A:=\set{-2^{-i}}{(\forall j) \: d(x,x_i) \leq (1+\var)\:d(x,x_j)} \cup \{0\}.
$$
This is a $\Pi^0_1$-condition, hence $A$ is computable from
$(x,\overline{\{x_n\}}_n)$ as a nonempty member of $\AA_-(\omega+1)$.

Let now $z:=\min^-_{\omega+1}(A)$ and notice that:
\begin{enumerate}
\item if $z=-2^{-i}$ then by the definition of $A$ we have $x_i \in
    \ProjvarC^+_{\IR^n} (x,\overline{\{x_n\}}_n)$;
\item if $z=0$ then $(\forall i) \: (\exists j) \: d(x,x_j) <
    \frac{d(x,x_i)}{1+\var}$ and inductively we can prove that for every
    $k$ there exists $n$ such that $d(x,x_n) < \frac{d(x,x_0)}{(1+\var)^k}$
    which implies $x \in \overline{\{x_n\}}_n$; hence
    $\ProjvarC^+_{\IR^n}(x,\overline{\{x_n\}}_n)= \{x\}$.
\end{enumerate}

We need to show that from $z$ and the original input
$(x,\overline{\{x_n\}}_n)$ we can compute an effective Cauchy sequence
$(y[s])_s$ converging to a point $y \in \ProjvarC^+_{\IR^n}
(x,\overline{\{x_n\}}_n)$. To compute $y[s]$ set $j_0=0$ and start a
recursive procedure which will stop after finitely many steps. Given $j_k$
use (the name of) $z$ to check whether there exists $i \leq j_k$ such that
$z= -2^{-i}$; in this case we stop the recursion. If instead $z \neq -2^{-i}$
for every $i \leq j_k$ it follows that $-2^{-j_k}\notin A$ and hence there
exists $j_{k+1}$ such that $d(x,x_{j_{k+1}}) < \frac{d(x,x_{j_k})}{1+\var}$.
Since this is a $\Sigma^0_1$ property, we can search for such a $j_{k+1}$
until we find one. The recursion will stop when either we find $i \leq j_k$
such that $z= -2^{-i}$ or we see that $d(x,x_{j_k}) <
\frac{2^{-s-3}}{1+\var}$ (if the first alternative never occurs, such a $k$
exists since $\lim_{k \to \infty} d(x,x_{j_k}) =0$ because $d(x,x_{j_k}) \leq
\frac{d(x,x_0)}{(1+\var)^k}$). In the first case let $y[s] := x_i[s+3]$,
while in the second case let $y[s] := x_{j_k}[s+3]$.

It is clear from the construction that if $z= -2^{-i}$ and $y[s] = x_i[s+3]$
we have $y[s'] = x_i[s'+3]$ for every $s' \geq s$. This implies that if for
some $s$ we have $y[s] = x_i[s+3]$ with $z= -2^{-i}$ then the sequence
$(y[s])_s$ converges to $x_i$, which belongs to $\ProjvarC^+_{\IR^n}
(x,\overline{\{x_n\}}_n)$ by (1) above. If instead for every $s$ the first
possibility never occurs it means that $z=0$, so that by (2)
$\ProjvarC^+_{\IR^n}(x,\overline{\{x_n\}}_n)= \{x\}$, and indeed $(y[s])_s$
converges to $x$.

However, the convergence of $(y[s])_s$ does not suffice, and we need to check
that we actually defined an effective Cauchy sequence. For this it suffices
to show that $d(y[s], y[s+1]) < 2^{-s-1}$ for every $s$. This is obvious if
$z= -2^{-i}$, $y[s] = x_i[s+3]$ and $y[s+1] = x_i[s+4]$. Now assume that
neither $y[s]$ nor $y[s+1]$ have been defined using $z= -2^{-i}$. In other
words, $y[s]= x_{j_k}[s+3]$ and $y[s+1]= x_{j_h}[s+4]$ where $d(x,x_{j_k}) <
\frac{2^{-s-3}}{1+\var}$ and $d(x,x_{j_h}) < \frac{2^{-s-4}}{1+\var}$. Then
\begin{align*}
 d(y[s], y[s+1]) & \leq d(x_{j_k}[s+3], x_{j_k}) + d(x_{j_k}, x) + d(x, x_{j_h}) + d(x_{j_h}, x_{j_h}[s+4]) \\
  & < 2^{-s-3} + \frac{2^{-s-3}}{1+\var} + \frac{2^{-s-4}}{1+\var} + 2^{-s-4} < 2^{-s-1}.
\end{align*}
The last possibility (by the observation above) is that $y[s]= x_{j_k}[s+3]$
with $d(x,x_{j_k}) < \frac{2^{-s-3}}{1+\var}$ and $y[s+1] = x_i[s+4]$ with
$z= -2^{-i}$. In this case notice that, since $-2^{-i} \in A$, we have
$d(x,x_i) \leq (1+\var)\:d(x,x_{j_k}) < 2^{-s-3}$. Then
\begin{align*}
 d(y[s], y[s+1]) & \leq d(x_{j_k}[s+3], x_{j_k}) + d(x_{j_k}, x) + d(x, x_i) + d(x_i, x_i[s+4]) \\
  & < 2^{-s-3} + \frac{2^{-s-3}}{1+\var} + 2^{-s-3} + 2^{-s-4} < 2^{-s-1}.\qedhere
\end{align*}
\end{proof}

\begin{theorem}\label{omega1var+}
For every $\var>0$ and $n \geq 1$, $\Sort \leqsW \ProjvarC^+_{\IR^n}$.
\end{theorem}
\begin{proof}
By Proposition \ref{sortomega} it suffices to show that $\min^-_{\omega+1}
\leqsW \ProjvarC^+_{\IR^n}$.

As usual, it suffices to show the reduction for $n=1$. Fix $b \in \IN$ such
that $1+\var < b$ and notice that $b \geq 2$. Given $A:= (\omega+1) \setminus
\bigcup_{i\in\IN} B_i$ closed and nonempty in $\omega+1$, with
$B_0,B_1,B_2,\dots$ rational open balls in $\omega+1$, we compute a sequence
$(x_n)_n$ in $\IR$ by setting $x_n:=-b^{-k-1}$ for the least $k$ such that
$-2^{-k} \notin \bigcup_{i\leq n}B_i$ if such a $k$ exists, and otherwise
setting $x_n:=0$.

If $\min(A)=0$, then $-2^{-k}\notin A$ for every $k\in\IN$, which implies
$0\in\overline{\{x_n\}}_n$. Hence $\ProjvarC^+_\IR(0,\overline{\{x_n\}}_n)
=\{0\}$.

If instead $\min(A)=-2^{-k}$ then $k \in \IN$ is the least natural number
such that $-2^{-k}\in A$ and $\overline{\{x_n\}}_n = \{x_n\}_n$ is a discrete
subset of the closed interval $[-\frac 1b,-b^{-k-1}]$. In fact, for all $n$,
$x_n=-b^{-i-1}$ for some $i \leq k$ and, for $n$ sufficiently large,
$x_n=-b^{-k-1}$. This implies that $d(0, \overline{\{x_n\}}_n) = b^{-k-1}$.
Moreover, if $x_n \neq -b^{-k-1}$ then $x_n = -b^{-i-1}$ for some $i<k$ and
hence
$$
(1+ \var)\: d(0,\overline{\{x_n\}}_n) < b\cdot d(0,\overline{\{x_n\}}_n) =
b^{-k} \leq b^{-i-1} = d(0,x_n)
$$
and thus $x_n \notin \ProjvarC^+_\IR(0,\overline{\{x_n\}})$. We thus showed
that $\ProjvarC^+_\IR(0,\overline{\{x_n\}}_n) = \{-b^{-k}\}$.

We have proved that $\ProjvarC^+_\IR(0,\overline{\{x_n\}}_n)$ is a singleton
and we now show that its unique element $y$ can be used to compute
$\min^-_{\omega+1}(A)$. Given then such $y \in \IR$, we produce the
$\rho_{\omega+1}$-name $q$ of $\min^-_{\omega+1}(A)$ in the following way.
For each $s$ we check whether $y = -b^{-s-1}$ or not. Notice that this test
is decidable because $y$ belongs to $\set{-b^{-n-1}}{n\in\IN} \cup \{0\}
\sbsq \IR$ and $-b^{-s-1}$ is not an accumulation point of this set. If the
answer is positive, we put $q(s)=1$, otherwise $q(s)=0$.
\end{proof}

\begin{corollary}\label{approxequivsort}
For every $\var>0$ and $n \geq 1$, $\ProjvarC^+_{\IR^n} \equivW \Sort$.
\end{corollary}
\begin{proof}
By Theorem \ref{var+omega1} and Theorem \ref{omega1var+}.
\end{proof}

\begin{corollary}
\label{corr:incomparable}
For every $\var>0$ and $n \geq 1$, $\ProjvarC^+_{\IR^n} \nW \ProjvarC^-_{\IR^n}$.
\end{corollary}
\begin{proof}
By Proposition \ref{prop:sortcr}, Theorem \ref{approxequivCR} and Corollary \ref{approxequivsort}.
\end{proof}


\subsection{Total approximated projection operators}

Our classification of projection operators ends finally with a computable
version of projection, that can be therefore used in concrete applications.

\begin{theorem}\label{totalapprox}
For every $\var>0$ and $n \geq 1$, $\ProjvarC_{\IR^n}$ is computable.
\end{theorem}
\begin{proof}
By Fact \ref{approxkridc}.1 we can assume that $\var$ is computable. We give
an algorithm to determine some $y\in \ProjvarC_{\IR^n}(x,A)$ for every
$(x,A)\in\IR^n\times\AA(\IR^n)$ with $A\neq\eps$.

We know already that total information on $A$ allows us to compute $d(x,A)$,
and then $(1+\var)\:d(x,A)$. We construct by induction an approximate
projection point of $x$ onto $A$ as follows.

At stage $s\geq 0$ we check whether
$$
\text{(\emph{i})} \:\:(1+\var)\:d(x,A)<2^{-s-3} \qquad \textrm{or} \qquad \text{(\emph{ii})} \:\:d(x,A)>0.
$$
Notice that at least one of these two conditions holds, and we stop when we
verify one of them. If (\emph{i}) is verified before (\emph{ii}), we let
$y[s]:=x[s+3]$, and move to step $s+1$. If instead (\emph{ii}) is verified
before (\emph{i}), we inspect the dense sequence in $A$ searching for some
$z$ such that $d(x,z)<(1+\var)\:d(x,A)$ (a suitable $z$ always exists in this
case) and then let $y[t]=z[t+1]$ for all $t\geq s$.

We now show that the algorithm works. First we need to check that $(y[i])_i$
is an effective Cauchy sequence converging to some $y$, and to this end it
suffices to check that $d(y[s],y[s+1]) \leq 2^{-s-1}$ for all $s$. This is
trivial if at stage $s$ and as well at stage $s+1$ the condition (\emph{i})
is satisfied first, or alternatively if (\emph{ii}) has been verified at some
stage $t\leq s$. The interesting case is therefore when (\emph{ii}) is
verified for the first time at stage $s+1$. Then
\begin{multline*}
 d(y[s],y[s+1]) = d(x[s+3],z[s+2]) < d(x[s+3],x) +d(x,z) +d(z,z[s+2]) \\
 < 2^{-s-3}+(1+\var)\:d(x,A) +2^{-s-2} < 2^{-s-1}.
\end{multline*}

We then need to check that $y \in \ProjvarC_{\IR^n} (x,A)$. If (\emph{i}) has
always been verified, then $d(1+\var)\:(x,A)=0= d(x,y)$, since $y=x$. If at
stage $s$ (\emph{ii}) is verified, then $y=z$ where $z$ was picked so that
$d(x,z) < (1+\var)\:d(x,A)$
\end{proof}

\section{An application: the Whitney Extension Theorem}\label{sec:Whitney}

Projection points are often used in mathematics. An example is the Whitney
Extension Theorem, originally proved in \cite{Whitney}, and dealing with
differentiable functions in $\IR^n$. This theorem considers a real-valued
continuous function $f$ defined on a closed $A \subseteq \IR^\IN$. Since $A$
is closed, we cannot even attempt to compute the partial derivatives of $f$
at many boundary points of $A$. However we can have also a set of continuous
functions (the pseudo-derivatives of $f$) defined also on $A$ which satisfy
Taylor's formulas and hence behave like the partial derivatives of degree
$\leq k$ of $f$ ($f$ and this set of functions are collectively called a
\emph{jet}). The Whitney Extension Theorem asserts that under these hypotheses
$f$ can be extended to some $g \in C^k(\IR^n)$, so that $g$ and its partial
derivatives extend the elements of the jet.

A classical proof of the Whitney Extension Theorem is contained in \cite[Chapter
VI]{Stein}, and we follow Stein's proof to provide a computable version.
Starting with the closed set $A$, Stein defines a family $\FF$ of cubes
tiling the complement of $A$ and a partition of unity $(\varphi^*_Q)_{Q \in
\FF}$ consisting of smooth functions. For each $Q \in \FF$ let $P_Q$ be a
projection point of the center of $Q$ onto $A$. We then define the $C^k$
extension of $f$ by
$$
g(x):=
\begin{cases}
f(x) & \textrm{if $x\in A$;}\\
\sum_{Q\in\FF}f(P_Q)\varphi^*_Q(x) & \textrm{if $x \notin A$.}
\end{cases}
$$
(Notice that, for a given $A$ and after we fix $\FF$, $(\varphi^*_Q)_{Q \in
\FF}$ and $(P_Q)_{Q \in \FF}$, in fact we obtain a linear operator from the
space of jets to $C^k(\IR^n)$.)

If $A$ is given with total information, variations of $\FF$ and
$(\varphi^*_Q)_{Q \in \FF}$ can be computed. Thus at first sight the only
essentially non-computable step (by Proposition \ref{LLPOProjC} and Theorem
\ref{WKLProjC}) in Stein's proof is the choice of $(P_Q)_{Q \in \FF}$. To
overcome this obstacle $P_Q$ can be replaced with some other point of $A$
which is close enough to $Q$, and \lq\lq close enough\rq\rq\ depends only
from the size of $Q$. This suggests that the multi-valued functions naturally
associated to the Whitney Extension Theorem are actually computable without
resorting to any projections. However there is another, subtler, point that
needs to be taken into account. In fact, in the definition by cases of $g$
given above the case distinction is not computable. Thus we need to provide
an effective way, given $x$ and $A$, to compute $g(x)$ without knowing
whether $x \in A$. Here the projection operators, which are defined over
$\IR^n$, come back into the picture and appear to be essential: when we do
not know positively that $x \notin A$ they are used to compute $g(x)$ in a
way that is compatible with both cases. Only by showing that approximate
projections are indeed sufficient it is possible to find a computable version
of the Whitney Extension Theorem.

Summing up, assuming $A$ is represented with total information and using
Theorem \ref{totalapprox}, we show that the multi-valued function associated
to the Whitney Extension Theorem is computable.
As mentioned in the introduction, full details of this result will be
included in a forthcoming paper (\cite{compWhitney}).

\bibliographystyle{alpha}
\bibliography{projectionsWhitney}{}

\end{document}